\documentclass{amsart}
\usepackage{amssymb}
\usepackage{mathrsfs}
\usepackage{enumerate}
 \usepackage{hyperref}
 \usepackage{cleveref}

\newtheorem{thm}{Theorem}
\newtheorem{defn}{Definition}
\newtheorem{fact}{Fact}
\newtheorem{lem}[thm]{Lemma}
\newtheorem{prop}[thm]{Proposition}
\newtheorem{cor}[thm]{Corollary}
\newtheorem{question}{Question}
\newtheorem{case}{Case}

\newcommand{\up}{\upharpoonright}
\newcommand{\ra}{\rightarrow}
\newcommand{\mc}{\mathcal}
\newcommand{\ms}{\mathscr}
\newcommand{\sm}{^\smallfrown}
\newcommand{\lo}{<_{ho}}
\newcommand{\m}{MA$_{\omega_1}(S)[S]$ }

\makeatletter
\@addtoreset{case}{thm}
\@addtoreset{case}{lem}
\@addtoreset{case}{prop}
\makeatother

\begin{document}

  \title{MA$_{\omega_1}(S)[S]$ does not imply $\mathcal{K}_2$}
  \author{Yinhe Peng}
  \address{Academy of Mathematics and Systems Science, Chinese Academy of Sciences\\ East Zhong Guan Cun Road No. 55\\Beijing 100190\\China}
\email{pengyinhe@amss.ac.cn}

  \author{Liuzhen Wu}
  \address{Academy of Mathematics and Systems Science, Chinese Academy of Sciences\\ East Zhong Guan Cun Road No. 55\\Beijing 100190\\China}
\email{lzwu@math.ac.cn}

\thanks{Peng was partially supported by NSFC No. 11901562 and a program of the Chinese Academy of Sciences. Wu was partially supported by NSFC No. 11871464.}


  \begin{abstract}
  We construct a model in which MA$_{\omega_1}(S)[S]$ holds and $\mathcal{K}_2$ fails. This shows that MA$_{\omega_1}(S)[S]$ does not imply $\mathcal{K}_2$ and answers an old question of Larson and Todorcevic in \cite{lt}. We also investigate different strong colorings in models of MA$_{\omega_1}(S)[S]$.

  \end{abstract}
  
    \maketitle

  \section{Introduction}
  The method of iterated forcing was introduced   by Solovay and Tennenbaum in \cite{ma} to prove the consistency of no Suslin trees. Then Martin formulated an axiom MA$_{\omega_1}$ whose consistency follows from Solovay-Tennenbaum's iterated forcing method.
  
  MA$_{\omega_1}$ is the following statement:
  \begin{itemize}
  \item If $\mc{P}$ is a c.c.c. poset and $\{D_\alpha: \alpha<\omega_1\}$ is a collection of dense subsets of $\mc{P}$, then there is a filter $G\subset \mc{P}$ such that $G\cap D_\alpha\neq \emptyset$ for any $\alpha<\omega_1$.
  \end{itemize}
  A partial order $\mc{P}$ has the \emph{countable chain condition} (c.c.c.)  if every pairwise incompatible subset of $\mc{P}$ is at most countable.

Since formulated, MA$_{\omega_1}$ plays an important role in proving consistency, distinguishing properties, etc. For example,    Galvin observes that MA$_{\omega_1}$ implies the failure of Pr$_1(\omega_1, 2, \omega)$ whose consistency follows from $\mathfrak{b}=\omega_1$ (see \cite{st89}). For $\theta\leq \omega$, Pr$_1(\omega_1, \kappa, \theta)$ is the following statement:
\begin{itemize}
  \item (\cite{ss94}) There is a function $c:[\omega_1]^2\rightarrow \kappa$ such that whenever we are given $n<\theta$, a collection $\langle a_\alpha: \alpha< \omega_1\rangle$ of pairwise disjoint elements of $[\omega_1]^{n}$ and  $\eta<\kappa$,   there are $\alpha< \beta$ such that $c( a_\alpha(i), a_\beta(j))=\eta$ for any $i, j<n$.
  \end{itemize}
  Here $a(i)$ is the $i$th element of $a$ in the increasing enumeration for $a\in [\omega_1]^n$ and $i<n$.
  It turns out that $\omega$ is the least possible value since Pr$_1(\omega_1, \omega_1, n)$ holds in ZFC for any $n<\omega$ by \cite{pw}. But this leaves open the equivalence among properties between MA$_{\omega_1}$ and the failure of Pr$_1(\omega_1, 2, \omega)$ (see, e.g., \cite{tv}).
  
  For example, a property that is closely related to both combinatorics and forcing was investigated in \cite{tv}. %A coloring $c: [\omega_1]^2\ra 2$ is \emph{c.c.c. destructible}  (see \cite{tv}) if $\omega_1$ is a countable union of 0-homogeneous sets in a c.c.c. forcing extension. 
  $\ms{K}_2$, introduced in \cite{tv}, is the assertion that one of the following two equivalent statements holds:
  \begin{itemize}
  \item Every partition $ [\omega_1]^2=K_0\cup K_1$ has an uncountable 0-homogeneous set whenever it has an uncountable 0-homogeneous set in a c.c.c. forcing extention.
  \item Every c.c.c. poset has property (K), i.e., every uncountable subset of a c.c.c. poset has an uncountable subset that is pairwise compatible.
  \end{itemize}
  
  Clearly, MA$_{\omega_1}\Rightarrow \ms{K}_2\Rightarrow \neg$Pr$_1(\omega_1, 2,\omega)$. And a test question for strength of properties between MA$_{\omega_1}$ and $\neg\mathrm{Pr}_1(\omega_1, 2, \omega)$ is the following question which was conjectured in \cite{tv} to have a negative answer.
  \begin{question}\label{q1}
  Does $\ms{K}_2$ imply $\mathrm{MA}_{\omega_1}$?
  \end{question}
  A strategy to distinguish properties between MA$_{\omega_1}$ and $\neg\mathrm{Pr}_1(\omega_1, 2, \omega)$ has been known  for more than twenty years. That is, to find a property with large fragments of MA$_{\omega_1}$ that does not imply MA$_{\omega_1}$ --- MA$_{\omega_1}(S)[S]$ (see \cite{lt}) or PFA$(S)[S]$ (see \cite{PFAS}). Then try to prove that the large fragments of  MA$_{\omega_1}$ imply, e.g.,  $\ms{K}_2$. Note that MA$_{\omega_1}$ fails in any forcing extension over a Suslin tree $S$ since $\Vdash_S \mathfrak{t}=\omega_1$.
  
  For a Suslin tree $S$, MA$_{\omega_1}(S)$ is the following assertion (see \cite{lt}):
  \begin{itemize}
  \item for any c.c.c. poset $\mathcal{P}$ with $\Vdash_\mc{P} ``S$ is Suslin'', for any collection $\{D_\alpha: \alpha<\omega_1\}$ of dense   subsets of $\mc{P}$, there is a filter $G\subset \mc{P}$ meeting them all.
  \end{itemize}
MA$_{\omega_1}(S)[S]$ holds if (see \cite{ltall}) the universe is a forcing extension by $S$ over a model of MA$_{\omega_1}(S)$ where $S$ is the Suslin tree. PFA$(S)$ and PFA$(S)[S]$ are defined in an analogous way.

The strategy may have potential to reveal more connections between properties concerning combinatorics and/or forcing on $\omega_1$. A general test question for the strategy which has its own interest is: what combinatorial properties follow from MA$_{\omega_1}(S)[S]$. One example is the pure combinatorial property $\mc{K}_2$ introduced in \cite{lt}.
\begin{itemize}
\item (\cite{st89}) A partition $[\omega_1]^2=K_0\cup K_1$ is a \emph{c.c.c. partition} if every uncountable family of finite 0-homogeneous sets contains two members whose union is also 0-homogeneous.
\item (\cite{lt}) $\mc{K}_2$ is the assertion that every c.c.c. partition of $[\omega_1]^2$ has an uncountable 0-homogeneous subset.
\end{itemize}
Clearly, MA$_{\omega_1}\Rightarrow \ms{K}_2\Rightarrow \mc{K}_2 \Rightarrow \neg$Pr$_1(\omega_1, 2,\omega)$.

A   question to test the strategy mentioned above and   to find out more combinatorial consequences of MA$_{\omega_1}(S)[S]$, is the following question asked by Larson and Todorcevic in \cite{lt}.
\begin{question}[\cite{lt}]\label{q2}
Does $\mathrm{MA}_{\omega_1}(S)[S]$ imply $\mc{K}_2$?\footnote{$\mathrm{MA}_{\omega_1}(S)$ was called $\mathrm{SA}_{\omega_1}$ in \cite{lt}.}
\end{question}

We answer Question \ref{q2} negatively in Section 4.
\begin{thm}\label{thm1}
$\mathrm{MA}_{\omega_1}(S)[S]$ is consistent with $\mathrm{Pr}_1(\omega_1, 2,\omega)$. In particular, $\mathrm{MA}_{\omega_1}(S)[S]$ does not imply $\mc{K}_2$.
\end{thm}
In particular, the attempt of  proving  ``MA$_{\omega_1}(S)[S]\Rightarrow \ms{K}_2$''  (thus answering Question \ref{q1} negatively) fails.\medskip

Another  closely related property  is Pr$_0(\omega_1, \omega, \omega)$.  For $\theta\leq \omega$, Pr$_0(\omega_1, \kappa, \theta)$ is the following statement:
\begin{itemize}
  \item (\cite{ss94}) There is a function $c:[\omega_1]^2\rightarrow \kappa$ such that whenever we are given $n<\theta$, a collection $\langle a_\alpha: \alpha< \omega_1\rangle$ of pairwise disjoint elements of $[\omega_1]^{n}$ and a function $h: n\times n\rightarrow \kappa$,   there are $\alpha< \beta$ such that $c( a_\alpha(i), a_\beta(j))=h(i,j)$ for any $i,j <n$.
  \end{itemize}
In \cite{st89}, Todorcevic proved that $\mathfrak{b}=\omega_1\Rightarrow $Pr$_0(\omega_1, \omega, \omega)$ and implicitly asked (see \cite[1.5]{st89}) the following question.
\begin{question}\label{qt}
Does $\mathfrak{t}=\omega_1$ imply $\mathrm{Pr}_0(\omega_1, \omega, \omega)$?
\end{question}
Since MA$_{\omega_1}(S)[S]\Rightarrow \mathfrak{t}=\omega_1$, a positive answer to the following question would give a negative answer to Question \ref{qt}.
\begin{question}\label{q4}
Does $\mathrm{MA}_{\omega_1}(S)[S]$ imply $\neg\mathrm{Pr}_0(\omega_1, \omega, \omega)$?
\end{question}
However, this attempt fails and we prove the following in Section 5.
\begin{thm}\label{thm2}
$\mathrm{MA}_{\omega_1}(S)[S]$ is consistent with $\mathrm{Pr}_0(\omega_1, \omega,\omega)$. 
\end{thm}

 However, whether the stronger property PFA$(S)[S]$ settles above questions remains open.
\begin{question}
Does $\mathrm{PFA}(S)[S]$ imply $\ms{K}_2$? or $\mc{K}_2$? or $\neg\mathrm{Pr}_0(\omega_1, \omega, \omega)$?
\end{question} \medskip

This paper is organized as follows. Section 3 provides the initial model that we are going to start with.   In Section 4, we introduce the property $\varphi_0(c)$ for a coloring $c: [S]^2\ra 2$ on $S$. Iterated forcing c.c.c. posets that preserve $\varphi_0(c)$, we construct a model of MA$_{\omega_1}(S)[S]+$Pr$_1(\omega_1, 2, \omega)$. Thus answer Question \ref{q2} negatively.  We then show that $\varphi_0(c)$ is rejected by PFA$(S)$. This distinguishes MA$_{\omega_1}(S)$ and PFA$(S)$ by a partition property on $S$. 
In Section 5, we introduce a new property $\varphi_1(c)$  and  construct a model of MA$_{\omega_1}(S)[S]+$Pr$_0(\omega_1, \omega_1, \omega)$.

  \section{preliminary}
In this section, we introduce some standard definitions and facts.

Suslin trees in this paper are subsets of $2^{<\omega_1}$, ordered by extension.

For a  Suslin tree $S$ and $t\in S$, the\emph{ height} of $t$ is $ht(t)=dom(t)$. In other words, $ht(t)$ is the order type of $\{s\in S: s<_S t\}$. For any $\alpha<\omega_1$, the $\alpha$th level of $S$ is $S_\alpha=\{t\in S: ht(t)=\alpha\}$. For any $s, t\in 2^{<\omega_1}$, 
$$\Delta(s,t)=\max \{\alpha\leq \min\{dom(s), dom(t)\}: s\up_\alpha = t\up_\alpha\},$$ 
$$s \wedge t=s\up_{\Delta(s,t)},$$
$$D_{s,t}=\{\xi<\min\{dom(s), dom(t)\}: s(\xi)\neq t(\xi)\}.$$

A Suslin tree $S$  is \emph{coherent} if  $D_{s,t}$ is finite for any $s, t\in S$. 

A coherent Suslin tree $S\subset 2^{<\omega_1}$   is \emph{complete} if for any $s\in S$, for any $t\in 2^{ht(s)}$, $t\in S$ whenever $|D_{s,t}|<\omega$.\medskip

$[X]^\kappa$ is the set of all subsets of $X$ of size $\kappa$. Moreover, if $X$ is a set of ordinals, $k<\omega$ and $b\in [X]^k$, then $b(0),b(1),...,b(k-1)$ is the increasing enumeration of $b$.% and    throughout this paper, we will use $(b(0),...,b(k-1))$ to denote $b$.

  \begin{defn}\label{ho}
  \begin{enumerate}
\item   The \emph{height order}, denoted by $<_{ho}$, is the linear order on  $2^{<\omega_1}$ defined by for any $s\neq t$ in  $2^{<\omega_1}$,  $s<_{ho} t$ if any of the following conditions holds:
\begin{itemize}
\item $ht(s)<ht(t)$;
\item $ht(s)=ht(t)$ and $ s(\Delta(s,t))<t(\Delta(s,t))$.
\end{itemize}
  %$<_{lex}$ is the \emph{lexicographical order} of $2^{<\omega_1}$, i,e., for any $s, t\in 2^{<\omega_1}$, $s\lo t$ if either $s$ is a proper initial segment of $t$ or $s$ is incomparable with $t$ and $s(\Delta(s,t))< t(\Delta(s,t))$.   
  % For $a\in [2^{<\omega_1}]^{<\omega}$, $a(0),a(1),...,a(k-1)$ is the $\lo$-increasing enumeration of $a$.   
  \item For $a\in [2^{<\omega_1}]^{<\omega}$, $a(0),a(1),...,a(|a|-1)$ is the $\lo$-increasing enumeration of $a$. 
  \end{enumerate}
   \end{defn}

\section{An initial model}
In this section, we find a model with a complete coherent Suslin tree $S$ and a coloring on $S$ witnessing Pr$_0(S, \omega, \omega)$.

%For a Suslin tree $S$, two sequences $a, b\in S^{<\omega}$ are \emph{range-disjoint} if their ranges are disjoint, i.e., $\{a(i): i<|a|\}\cap \{b(j): j<|b|\}=\emptyset$.  $a\in S^{<\omega}$ is a \emph{1-1 sequence} if $a(i)\neq a(j)$ for any $i\neq j$ in $dom(a)$.

\begin{defn}
Suppose that $S$ is a Suslin tree.
\begin{enumerate}
\item $\mathrm{Pr}_1(S, \kappa, \theta)$ is the assertion that there is a function $c: [S]^2\ra \kappa$ such that whenever we are given $n<\theta$, a collection $\langle s_\alpha: \alpha<\omega_1\rangle$ of pairwise distinct elements of $S$,  a collection $\langle a_\alpha: \alpha<\omega_1\rangle$ of pairwise disjoint elements  of $[S]^n$ and $\eta<\kappa$,   there are $\alpha<\beta$ such that $s_\alpha<_S s_\beta$ and $c( a_\alpha(i), a_\beta(j))=\eta$ for any $i,j <n$.

\item  $\mathrm{Pr}_0(S, \kappa, \theta)$ is the assertion that there is a function $c: [S]^2\ra \kappa$ such that whenever we are given $n<\theta$, a collection $\langle s_\alpha: \alpha<\omega_1\rangle$ of pairwise distinct elements of $S$,  a collection $\langle a_\alpha: \alpha<\omega_1\rangle$ of pairwise disjoint elements  of $[S]^n$ and $h: n\times n\ra\kappa$,    there are $\alpha<\beta$ such that $s_\alpha<_S s_\beta$ and $c( a_\alpha(i), a_\beta(j))=h(i,j)$ for any $i,j <n$.
\end{enumerate}
\end{defn}
Throughout the paper, $\theta\leq \omega$.

A standard lifting argument shows that Pr$_1(S, \omega, \omega)$ implies Pr$_1(S, \omega_1, \omega)$ (see, e.g., the beginning of Section 4 of \cite{st87}) and Pr$_1(S, \omega_1, \omega)$ implies Pr$_0(S, \omega_1, \omega)$ (see, e.g., Lemma 1 of \cite{ss91}).

The following construction is standard. For completeness, we sketch a proof.

\begin{fact}[$\diamondsuit$]\label{fact1}
There is a complete coherent Suslin tree $S\subset 2^{<\omega_1}$ and a strong coloring $c: [S]^2\ra \omega$ witnessing $\mathrm{Pr}_0(S, \omega, \omega)$.
\end{fact}
\begin{proof}[Sketch proof.]
Fix a $\diamondsuit$-sequence $\{C_\alpha: \alpha<\omega_1\}$. We first construct $S_\alpha$ by induction on $\alpha$ such that each $S_\alpha$ is complete coherent. 

We shall only deal with the case that $\alpha$ is a limit ordinal and $C_\alpha$ is a maximal antichain of $S\up_\alpha$. Let $\{F_n: n<\omega\}$ enumerate $[\alpha]^{<\omega}$ and $\{\alpha_n: n<\omega\}$ be cofinal in $\alpha$.
We will construct $\langle s^n_m\in 2^{<\alpha}: m\leq n\rangle$ by induction on $n$. 

At step $n=2^i(2j+1)$, first choose $s\in S_{ht(s^{n-1}_i)}$ such that $D_{s, s^{n-1}_i}=F_j$ (extend $s^{n-1}_m$ first if necessary). Then extend $s$ to some $s^n_n\in S\up_\alpha$ such that $s^n_n$ extends some node in $C_\alpha$ and $ht(s^n_n)>\alpha_n$. Finally for each $m<n$, $s^n_m={s^{n-1}_m}^\smallfrown s^n_n\up_{[ht(s), ht(s^n_n))}$.

For each $m<\omega$, $s_m=\bigcup_{n\geq m} s^n_m$. Then $S_\alpha=\{s_m: m<\omega\}$ is complete coherent. Now $S=\bigcup_{\alpha<\omega_1} S_\alpha$ is a complete coherent Suslin tree.\medskip

We then construct, using CH, a strong coloring $c:[S]^2\ra \omega$ witnessing Pr$_1(S, \omega, \omega)$. Then a standard lifting will induce a coloring witnessing Pr$_0(S, \omega, \omega)$.\footnote{See, e.g., the beginning of Section 5.}

Let $\{\mc{A}_\alpha: \alpha<\omega_1\}$ enumerate all $(n,\xi, \langle s_\beta: \beta<\xi\rangle, \langle a_\beta: \beta<\xi\rangle, k)$ such that
\begin{itemize}
\item $0<n<\omega, \xi<\omega_1, k< \omega$;
\item $\langle s_\beta: \beta<\xi\rangle$ is a collection of pairwise distinct elements of $S$;
\item $\langle a_\beta: \beta<\xi\rangle$ is a collection of pairwise  disjoint elements of $[S]^n$;
\item if $\mc{A}_\alpha=(n,\xi, \langle s_\beta: \beta<\xi\rangle, \langle a_\beta: \beta<\xi\rangle, k)$, then for any $\beta<\xi$, $s_\beta\in S\up_\alpha$ and $a_\beta\in [S\up_\alpha]^n$;

%\item if $\mc{A}_\alpha=(n,\xi, \langle s_\beta: \beta<\xi\rangle, \langle a_\beta: \beta<\xi\rangle, k)$ and $0<m<n$, then for some $\gamma<\alpha$, $\mc{A}_\gamma=(m, \xi, \langle s_\beta: \beta<\xi\rangle, \langle a_\beta\up_m: \beta<\xi\rangle, k)$.
\end{itemize}
Fix a well-ordering $<_w$ of $S$ of order type $\omega_1$ such that for any $s, t\in S$, $s<_w t$ whenever $ht(s)<ht(t)$. Denote $S_{< s}=\{t\in S: t<_w s\}$.

Now define $c\up_{S_{< s}\times \{s\}}$ by induction on $<_w$ such that
\begin{enumerate}[(i)$_s$]
\item for any $t\in S_{ht(s)}$,    $\gamma<ht(s)$ and $b\in [\{u\in S_{< s}: ht(u)\geq \gamma\}]^{<\omega}$, if $\mc{A}_\gamma=(n,\xi, \langle s_\beta: \beta<\xi\rangle, \langle a_\beta: \beta<\xi\rangle, k)$  and there are infinitely many $\beta$ such that 
\begin{itemize}
\item $s_\beta<_S t$;
\item $c(x, u)=k$ for any $x\in a_\beta$ and $u\in b$,
\end{itemize}
then there are infinitely many $\beta$ such that 
\begin{itemize}
\item $s_\beta<_S t$;
\item $c(x, u)=k$ for any $x\in a_\beta$ and $u\in b\cup \{s\}$.
\end{itemize}
\end{enumerate}
Since there are at most countably many such $t, \gamma$ and $b$, a standard $\omega$-step induction produces the desired                        $c\up_{S_{< s}\times \{s\}}$.

Now we prove that $c=\bigcup_{s\in S} c\up_{S_{< s}\times \{s\}}$ witnesses Pr$_1(S, \omega, \omega)$. Fix $0<n<\omega$, a collection $\langle s_\alpha: \alpha<\omega_1\rangle$ of pairwise distinct elements of $S$,  a collection $\langle a_\alpha: \alpha<\omega_1\rangle$ of pairwise disjoint elements  of $[S]^n$ and $k<\omega$. Extending $s_\alpha$ if necessary, we may assume that $ht(s_\alpha)\geq \max\{ht(u): u\in a_\alpha\}$ for any $\alpha$.
%Going to uncountable subsets, we may assume that for any $\alpha<\beta$, $\max\{ht(x): x\in a_\alpha\}<\min \{ht(y): y\in a_\beta\}$.

Choose $\alpha<\omega_1$  such that for some $\gamma, \xi<\min\{ht(a_\alpha(0)), \alpha\}$,% where for any $j<n$, $a_\alpha(j)$ is the $j$th element of $a_\alpha$ in $<_w$-increasing enumeration,
\begin{enumerate}
\item $\mc{A}_\gamma=(n,\xi, \langle s_\beta: \beta<\xi\rangle, \langle a_\beta: \beta<\xi\rangle, k)$;
\item there are infinitely many $\beta<\xi$ such that $s_\beta<_S s_\alpha$.
\end{enumerate}
Let $u_0,..., u_{n-1}$ be an $<_w$-increasing enumeration of $a_\alpha$.
Applying (i)$_{u_j}$, by induction on $j$,  to $s_\alpha\up_{ht(u_j)}, \gamma$ and $ b_j=\{ u_{j'}: j'<j\}$, we get (infinitely many) $\beta<\xi$ such that 
\begin{itemize}
\item $s_\beta<_S s_\alpha$;
\item $c(x,u)=k$ for any $x\in a_\beta, u\in a_\alpha$.
\end{itemize}
Then $\beta<\alpha$ are as desired.
\end{proof}
Our initial models will satisfy $2^\omega=\omega_1$, $2^{\omega_1}=\omega_2$ and contain a complete coherent Suslin tree $S$ with a strong coloring witnessing Pr$_0(S, \omega, \omega)$.

\section{\texorpdfstring{MA$_{\omega_1}(S)[S]$}{MAomega1(S)[S]} does not  imply \texorpdfstring{$\mc{K}_2$}{K2}}
In this section, we will construct a model of   \m  in which Pr$_1(\omega_1, 2, \omega)$ holds. In particular, $\mc{K}_2$ fails.

We will start with a  ground model $V$ in which  GCH holds   and 
\begin{itemize}
\item $S\subset 2^{<\omega_1}$ is   a complete coherent Suslin tree;
\item  $c: [S]^2\ra 2$ is a strong coloring  witnessing Pr$_0(S, 2, \omega)$.
\end{itemize}\medskip

The coloring in the final model witnessing Pr$_1(\omega_1, 2, \omega)$ will be induced from $c$. I.e., define $c^*: [\omega_1]^2\ra 2$ by $c^*(\alpha, \beta)=c(b(\alpha), b(\beta))$ for any $\alpha<\beta$ where $b$ is the generic branch of $S$. In order for $c^*$ to witness Pr$_1(\omega_1, 2, \omega)$,  we need $c$ to satisfy the following property in the model of MA$_{\omega_1}(S)$:
\begin{enumerate}[(I)]
\item for any $n<\omega$, for any collection $\langle s_\alpha: \alpha<\omega_1\rangle$ of pairwise distinct elements of $S$, for any collection $\langle a_\alpha\in [ht(s_\alpha)]^n: \alpha< \omega_1\rangle$ of pairwise disjoint elements,  for any $k<2$, there are $\alpha< \beta$ such that 
$$s_\alpha<_S s_\beta \text{ and }c( s_\alpha\up_{a_\alpha(i)}, s_\beta\up_{a_\beta(j)})=k\text{ for any }i,j <n.$$
\end{enumerate}
Although  property (I) is true in the ground model $V$, it is difficult to be preserved in the iteration   process. So, instead of preserving property (I), we will introduce a stronger property that will be preserved in the iterated forcing.\medskip

Before introducing the stronger property, we describe the procedure first.
\begin{enumerate}[($\bigstar1$)]
\item First, we introduce a stronger property $\varphi_0(c)$ and prove that it is true in $V$.

\item Then, we prove that $\varphi_0(c)$ is preserved under finite support iterated c.c.c. forcing. And we iterated force all c.c.c. posets of size $\leq \omega_1$ that preserve both $S$ and $\varphi_0(c)$.

\item Finally, we prove that after an $\omega_2$ step iteration, every c.c.c. poset that preserves $S$ preserves $\varphi_0(c)$ automatically. So we get a model of MA$_{\omega_1}(S)$ in which $\varphi_0(c)$ (and hence property (I)) holds.
\end{enumerate}\medskip

To introduce $\varphi_0(c)$, we will need the following notations.
\begin{defn}\label{def2}
For $0<n<\omega$, $\mc{C}^0_n$ is the collection of all uncountable $A\subset [S]^{2n}$ such that for some $s\in S$,     denoted by $stem(A)$,
\begin{enumerate}
\item for any $a\in A$, $s^\smallfrown 0<_S a(0)<_S...<_S a(n-1)$ and $s^\smallfrown 1<_S a(n)<_S...<_S a(2n-1)$;\footnote{Recall that in Definition \ref{ho}, $a(i)$ is the $i$th element of $a$ in $<_{ho}$-increasing enumeration.}
\item for any $a\in A$, $ht(a(n-1))<ht(a(n))$;
\item for any $a\in A$, $D_{a(n-1), a(n)}=\{ht(s)\}$;
\item for any $a\neq b$ in $A$, either $ht(a(2n-1))< ht(b(0))$ or $ht(b(2n-1))<ht(a(0))$.
\end{enumerate}
\end{defn}
Note that for any $0<n<\omega$ and $A\in \mc{C}^0_n$, $stem(A)=a(0)\wedge a(n)$ for any $a\in A$.

\begin{defn}
\begin{enumerate}
\item $\mc{C}^0=\bigcup_{0<n<\omega} \mc{C}^0_n$.

\item For $A\in \mc{C}^0$, $N_A$ is the $n$ such that $A\in \mc{C}^0_n$. 
\end{enumerate}
\end{defn}
Now, we are ready to define the stronger property $\varphi_0$.

\begin{defn}
\begin{enumerate}
\item For $0<n<\omega$ and $\langle A_i\in \mc{C}^0: i<n\rangle$, $\varphi_0(c, A_0,..., A_{n-1})$ is the assertion that for any collection $\langle s_\alpha: \alpha<\omega_1\rangle$ of pairwise distinct elements of $S$,  for any $\langle x^i_\alpha\in [A_i]^{<\omega}: i<n, \alpha<\omega_1\rangle$ such that $\langle x^i_\alpha: \alpha<\omega_1\rangle$ is pairwise disjoint for each $i<n$, there are $\alpha<\beta$ such that
\begin{enumerate}[$(i)$]
\item $s_\alpha<_S s_\beta$;
\item for any $i<n$, for any $a\in x^i_\alpha$ and $b\in x^i_\beta$, there is $j<2$ such that for any $k, k'<N_{A_i}$,
$$c(a(k), b(k'))=j \text{ and } c(a(N_{A_i}+k), b(N_{A_i}+k'))=1-j.$$
\end{enumerate}

\item $\varphi_0(c)$ is the assertion that for any $0<n<\omega$, for any $\langle A_i\in \mc{C}^0: i<n\rangle$, $\varphi_0(c, A_0,..., A_{n-1})$ holds.
\end{enumerate}
\end{defn}
\textbf{Remark.} Repetition is allowed in   $\varphi_0(c, A_0,..., A_{n-1})$.

The following  is straightforward.
\begin{fact}
$\varphi_0(c)$ implies that $S$ is Suslin. 
\end{fact}

The proof that $\varphi_0(c)$ implies property (I) is simple.  But in order to have a sense of the property $\varphi_0(c)$, we will provide a proof as a warm up. We first recall   a simple  fact about Suslin trees. 
\begin{fact}\label{f2}
For any $X\in [S]^{\omega_1}$, there is $s\in S$ such that $X$ is dense above $s$, i.e., for any $t\geq_S s$, there is $x\in X$ such that $t\leq_S x$,.
\end{fact}
\begin{proof}
Note that for any  $s\in S$, there is $t\geq_S s$ such that $X$ is either dense above $s$,  or   empty above $t$, i.e., no $x\in X$ is above $t$. So take a maximal antichain $A\subset S$ such that for any $s\in A$, $X$ is either dense above $s$ or empty above $s$. Since $A$ is at most countable, choose $x\in X$ such that $ht(x)>\sup\{ht(s): s\in A\}$. Then choose $s\in A$ such that $s<_S x$. Clearly, $X$ is dense above $s$.
\end{proof}
\textbf{Remark.} Above argument actually shows the following: for any $X\in [S]^{\omega_1}$, there is $\alpha<\omega_1$ such that $X$ is dense above $x\up_\alpha$ for any $x\in X\setminus S\up_\alpha$.

\begin{lem}\label{lem1}
Assume $\varphi_0(c)$. For any $n<\omega$, for any collection $\langle s_\alpha: \alpha<\omega_1\rangle$ of pairwise distinct elements of $S$, for any collection $\langle a_\alpha\in [ht(s_\alpha)]^n: \alpha< \omega_1\rangle$ of pairwise disjoint elements   and  $k<2$, there are $\alpha< \beta$ such that $s_\alpha<_S s_\beta $ and $c( s_\alpha\up_{a_\alpha(i)}, s_\beta\up_{a_\beta(j)})=k$ for any $i,j <n$.
\end{lem}
\begin{proof}
Fix $0<n<\omega$, pairwise distinct $\langle s_\alpha: \alpha<\omega_1\rangle$,  pairwise disjoint $\langle a_\alpha\in [ht(s_\alpha)]^n: \alpha< \omega_1\rangle$ and $k<2$.

By Fact \ref{f2}, fix $s\in S$ such that $\{s_\alpha: \alpha<\omega_1\}$ is dense above $s$. For each $\alpha<\omega_1$, choose $g(\alpha)>f(\alpha)>\alpha$ such that
\begin{enumerate}
\item $s^\smallfrown 0<_S s_{f(\alpha)}$ and $s^\smallfrown 1<_S s_{g(\alpha)}$;
\item $a_{g(\alpha)}(0)>ht(s_{f(\alpha)})$ and $a_{f(\alpha)}(0)>ht(s)+1$;

\item $D_{s_{f(\alpha)}, s_{g(\alpha)}}=\{ht(s)\}$.
\end{enumerate}
We describe the procedure of finding $f(\alpha)$ and $ g(\alpha)$. Since $\{s_\alpha: \alpha<\omega_1\}$ is dense above $s$,  there are uncountably many $\xi$ such that $s^\smallfrown 0<_S s_\xi$. So we can find $f(\alpha)>\alpha$ such that  $s^\smallfrown 0<_S s_{f(\alpha)}$ and $a_{f(\alpha)}(0)>ht(s)+1$. Then by complete coherence of $S$, $t\triangleq s^\smallfrown 1^\smallfrown s_{f(\alpha)}\up_{[ht(s)+1, ht(s_{f(\alpha)}))}\in S$. By dense above $s$ property again, find $g(\alpha)>f(\alpha)$ such that $t<_S s_{g(\alpha)}$ and $a_{g(\alpha)}(0)>ht(s_{f(\alpha)})$.\medskip

Going to an uncountable subset, we may assume that for any $\alpha<\beta$,
\begin{enumerate}\setcounter{enumi}{3}
\item $ht(s_{g(\alpha)})<a_{f(\beta)}(0)$.
\end{enumerate}
Then $A=\{b_\alpha: \alpha<\omega_1\}\in \mc{C}^0_n$ where for any $i<n$, 
\begin{enumerate}\setcounter{enumi}{4}
\item $b_\alpha(i)=s_{f(\alpha)}\up_{a_{f(\alpha)}(i)}$ and $b_\alpha(n+i)=s_{g(\alpha)}\up_{a_{g(\alpha)}(i)}$.
\end{enumerate}

Applying $\varphi_0(c)$ to $\langle s_{g(\alpha)}: \alpha<\omega_1\rangle$ and $\langle \{b_\alpha\}: \alpha<\omega_1\rangle$, we get $\alpha<\beta$ and $j<2$ such that
\begin{enumerate}\setcounter{enumi}{5}
\item $s_{g(\alpha)}<_S s_{g(\beta)}$;
\item for any $i,i'<n$,  $$c(b_\alpha(i), b_\beta(i'))=j \text{ and } c(b_\alpha(n+i), b_\beta(n+i'))=1-j.$$
\end{enumerate}
We claim that either $f(\alpha)< f(\beta)$ or $g(\alpha)< g(\beta)$ witnesses the conclusion of the lemma. Just note by (2), (3) and  (6),  $s_{f(\alpha)}<_S s_{f(\beta)}$.
\end{proof}

It follows immediately that forcing with $S$ over a model satisfying $\varphi_0(c)$ produces Pr$_1(\omega_1, 2, \omega)$.
\begin{cor}\label{cor}
Suppose that $c: [S]^2\ra 2$ is a coloring on a complete coherent Suslin tree $S$ such that $\varphi_0(c)$ holds. Then $\Vdash_S \mathrm{Pr}_1(\omega_1, 2, \omega)$.
\end{cor}
\begin{proof}
Let $\dot{c}^*: [\omega_1]^2\ra 2$ be  an $S$-name induced from $c$, i.e., 
$$s\Vdash_S \dot{c}^*(\alpha, \beta)=c(s\up_\alpha, s\up_\beta)$$
 for any $\alpha<\beta<\omega_1$ and $s\in S\setminus S\up_\beta$.
 
Then $\Vdash_S \dot{c}^* $ witnesses Pr$_1(\omega_1, 2, \omega)$. To see this, fix $t\in S$, $n<\omega$, an $S$-name $\dot{X}$ of an uncountable pairwise disjoint subset of $[\omega_1]^n$ and $k<2$. It suffices to find $s>_S t$ and disjoint $a, b\in [\omega_1]^n$ such that 
$$s\Vdash_S a, b\in \dot{X} \text{ and }\dot{c}^*(a(i), b(j))=k\text{ for any }i,j<n.$$
Find, for each $\alpha<\omega_1$, $s_\alpha>_S t$ and $a_\alpha\in [\omega_1\setminus \alpha]^n$ such that $ht(s_\alpha)>\max a_\alpha$ and
$$s_\alpha\Vdash_S a_\alpha\in \dot{X}.$$
Find $\Gamma\in [\omega_1]^{\omega_1}$ such that $\langle s_\alpha: \alpha\in \Gamma\rangle $ is pairwise distinct and $\langle a_\alpha: \alpha\in \Gamma\rangle$ is pairwise disjoint. Then by Lemma \ref{lem1}, there are $\alpha<\beta$ in $\Gamma$ such that  $s_\alpha<_S s_\beta $ and $c( s_\alpha\up_{a_\alpha(i)}, s_\beta\up_{a_\beta(j)})=k$ for any $i,j <n$. Then clearly $s_\beta\Vdash_S \dot{c}^*(a_\alpha(i), a_\beta(j))=k$ for any $i,j<n$.
\end{proof}

Now we prove that $\varphi_0(c)$ holds in $V$. In fact, it follows from Pr$_0(S, 2, \omega)$.
\begin{lem}\label{initial}
If $c$ witnesses $\mathrm{Pr}_0(S, 2, \omega)$, then $\varphi_0(c)$ holds.
\end{lem}
\begin{proof}
Arbitrarily choose $0<n<\omega$, $\langle A_i\in \mc{C}^0: i<n\rangle$,  a collection $\langle s_\alpha: \alpha<\omega_1\rangle$ of pairwise distinct elements of $S$,  a collection $\langle x^i_\alpha\in [A_i]^{<\omega}: i<n, \alpha<\omega_1\rangle$ such that $\langle x^i_\alpha: \alpha<\omega_1\rangle$ is pairwise disjoint for each $i<n$. It suffices to find $\alpha<\beta$ witnessing the assertion of $\varphi_0(c, A_0,..., A_{n-1})$.

%Going to uncountable subsets, we may assume that for some $\langle m_i: i<n\rangle$, $|x^i_\alpha|=m_i$ for any $i<n$ and $\alpha<\omega_1$.  

Now for each $\alpha<\omega_1$,  define
$$F_\alpha=\{a(k):  \text{ for some }i<n, a\in x^i_\alpha\text{ and } k<2N_{A_i}\} .$$
Denote 
$$D=\{ht(stem(A_i)): i<n\}.$$
Going to uncountable subsets, we may assume that 
\begin{enumerate}
\item $\{F_\alpha: \alpha<\omega_1\}$ is pairwise disjoint;
\item for any $\alpha<\omega_1$, $|F_\alpha|=|F_0|$ and $\max D< \min\{ht(s): s\in F_\alpha\}$;
%\item for any $i<n$ and $j<m_i$, the position of $x^i_\alpha(j)$ in $F_\alpha$ is independent of $\alpha$, i.e., for any $k<2N_{A_i}$ and $l<|F_\alpha|$,  $x^i_\alpha(j)(k)=F_\alpha(l)$ iff $x^i_0(j)(k)=F_0(l)$;
\item for any $\alpha<\omega_1$ and $k<|F_0|$, $F_\alpha(k)\up_D=F_0(k)\up_D$.
\end{enumerate}

Now define $h: |F_0|\ra 2$ by  for any $k<|F_0|$,
 $$h(k)=0 \text{ iff } |\{\xi\in D: F_0(k)(\xi)=0\}| \text{ is even.}$$ Apply Pr$_0(S, 2, \omega)$ to get $\alpha<\beta$ such that
\begin{enumerate}\setcounter{enumi}{3}
\item $s_\alpha<_S s_\beta$;
\item for any $k,k'< |F_0|$, $c(F_\alpha(k), F_\beta(k'))=h (k)$.
\end{enumerate}\medskip

We now check that $\alpha<\beta$ are as desired. First note that by (2), (3) and (5),
\begin{enumerate}\setcounter{enumi}{5}
\item for any $u\in F_\alpha$ and $v\in F_\beta$, $c(u,v)=0$ iff $|\{\xi\in D: u(\xi)=0\}|$ is even.
\end{enumerate}
Fix $i<n$, $a\in x^i_\alpha$, $b\in x^i_\beta$ and $k, k'<N_{A_i}$. It is clear from the definition of $\mc{C}^0$ and the fact that $ht(stem(A_i))\in D$ that
$$|\{\xi\in D: a(k)(\xi)=0\}|=|\{\xi\in D: a(0)(\xi)=0\}|\text{ and}$$
$$|\{\xi\in D: a(N_{A_i}+k)(\xi)=0\}|=|\{\xi\in D: a(0)(\xi)=0\}|-1.$$
Together with (6), for $j=c(a(0), b(0))$,
$$c(a(k), b(k'))=j \text{ and }c(a(N_{A_i}+k), b(N_{A_i}+k'))=1-j.$$
This, together with (4), shows that $\alpha<\beta$ are as desired and finishes the proof of the lemma.
\end{proof}

We then prove that $\varphi_0(c)$ is preserved at limit stages of finite support iteration of c.c.c. posets.
\begin{lem}\label{iteration}
Suppose that $\nu$ is a limit ordinal and $\mc{P}_\nu$ is the direct limit of the finite support iteration  $\langle \mc{P}_\beta, \dot{\mc{Q}}_\beta: \beta<\nu\rangle$   of c.c.c.  posets. If for any $\beta<\nu$, $\Vdash_{\mc{P}_\beta} \varphi_0(c)$, then $\Vdash_{\mc{P}_\nu} \varphi_0(c)$.
\end{lem}
\begin{proof}
Suppose otherwise. Fix $p\in \mc{P}_\nu$ and $0<n<\omega$, $\mc{P}_\nu$-names $\dot{A}_0,...,\dot{A}_{n-1}$ in $\dot{\mc{C}}^0$ such that
$$p\Vdash_{\mc{P}_\nu} \varphi_0(c, \dot{A}_0,...,\dot{A}_{n-1}) \text{ fails.}$$
Assume moreover that for $\mc{P}_\nu$-names $\dot{f}: \omega_1\ra S$ and $\dot{g}_i: \omega_1\ra [\dot{A}_i]^{<\omega}$ where $i<n$,
\begin{enumerate}
\item $p\Vdash_{\mc{P}_\nu} \langle \dot{f}(\alpha): \alpha<\omega_1\rangle \text{ and } \langle \dot{g}_i(\alpha): i<n, \alpha<\omega_1\rangle$   witness the failure of $\varphi_0(c, \dot{A}_0,...,\dot{A}_{n-1})$.
\end{enumerate}

Extending $p$ if necessary,   find $t_0,..., t_{n-1}$ and $m_0,...,m_{n-1}$ such that
\begin{enumerate}\setcounter{enumi}{1}
\item $p\Vdash_{\mc{P}_\nu} stem(\dot{A}_i)=t_i \text{ and } N_{\dot{A}_i}=m_i \text{ for any } i<n.$
\end{enumerate}

Now for any $\alpha<\omega_1$, find $p_\alpha<_{\mc{P}_\nu} p$  and $s_\alpha\in S$, $\langle x^i_\alpha: i<n \rangle$ such that 
\begin{enumerate}\setcounter{enumi}{2}
\item $p_\alpha\Vdash_{\mc{P}_\nu} \dot{f}(\alpha)=s_\alpha \text{ and } \dot{g}_i(\alpha)=x^i_\alpha \text{ for any } i<n.$
\end{enumerate}

Since $supp(p_\alpha)$ -- the support of $p_\alpha$ --  is finite for any $\alpha$, there is $\Gamma\in [\omega_1]^{\omega_1}$ such that $\{supp(p_\alpha): \alpha\in \Gamma\}$ forms a $\Delta$-system with root $r$. 

Choose $\eta<\nu$ such that $supp(p)\cup r\subset \eta$. Choose a generic filter $G$ of $\mc{P}_\eta$ containing $p\up_\eta$ such that $\Gamma'=\{\alpha\in \Gamma: p_\alpha\up_\eta\in G\}$ is uncountable.\medskip

Now work in $V[G]$. Find $\Gamma''\in[\Gamma']^{\omega_1}$ such that  $\langle s_\xi: \xi\in \Gamma''\rangle$ is pairwise distinct and for any $\alpha<\beta$ in $\Gamma''$,
\begin{enumerate}\setcounter{enumi}{3}
\item $ \max\{ht(s): s\in \bigcup\{ \bigcup x^i_\alpha: i<n\}\}< \min\{ht(s): s\in \bigcup\{ \bigcup x^i_\beta: i<n\}\}.$
\end{enumerate}
To see the existence of such $\Gamma''$, note that by c.c.c. of $\mc{P}_\nu$ and the fact that 
$$p\Vdash_{\mc{P_\nu}} rang(\dot{g}_i)\text{ is  pairwise disjoint for any }i<n,$$
 for any $\alpha<\omega_1$,   $\{\beta: \bigcup\{\bigcup x^i_\beta: i<n\}\cap S\up_\alpha\neq\emptyset\}$ is at most countable.

For $i<n$, let
$$B_i=\{a: a\in x^i_\alpha \text{ for some } \alpha\in \Gamma''\}.$$
It is straightforward to check that $B_i\in \mc{C}^0$ (in $V[G]$) for any $i<n$. 

Recall that $V[G]\vDash \varphi_0(c)$. Applying $\varphi_0(c)$ to $\langle s_\alpha: \alpha\in \Gamma''\rangle$ and $\langle x^i_\alpha: i<n, \alpha\in \Gamma''\rangle$, we find $\alpha<\beta$ in $\Gamma''$  such that
\begin{enumerate}\setcounter{enumi}{4}
\item $s_\alpha<_S s_\beta$;
\item for any $i<n$, for any $a\in x^i_\alpha$ and $b\in x^i_\beta$, there is $j<2$ such that for any $k,k'< m_i$,
$$c(a(k), b(k'))=j \text{ and } c(a(m_i+k), b(m_i+k'))=1-j.$$
\end{enumerate}\medskip

Now work in $V$. Note that $p_\alpha$ is compatible with $p_\beta$. Pick $p'<_{\mc{P}_\nu} p_\alpha, p_\beta$. But then (3), (5) and (6) show that
\begin{enumerate}\setcounter{enumi}{6}
\item $p'\Vdash_{\mc{P}_\nu}  \varphi_0(c, \dot{A}_0,...,\dot{A}_{n-1})$ is true for the assignments $ \dot{f}, \langle \dot{g}_i: i<n\rangle$ witnessed by $\alpha<\beta.$
\end{enumerate}
This contradicts (1) and the fact that $p'<_{\mc{P}_\nu} p$.
\end{proof}

\begin{defn}\label{defQ}
Suppose that $0<n<\omega$, $A_0,..., A_{n-1}$ are in $\mc{C}^0$ and $\langle x^i_\alpha\in [A_i]^{<\omega}:\alpha<\omega_1\rangle$ is pairwise disjoint for each $i<n$. Then $\mc{Q}_{\langle x^i_\alpha: i<n,\alpha<\omega_1\rangle}$ is the poset consisting of all $p\in [\omega_1]^{<\omega}$ such that
\begin{itemize}
\item  for any $\alpha<\beta$ in $p$ and $i<n$, for any  $a\in x^i_\alpha$ and  $b\in x^i_\beta$,  there is $j<2$ such that for any $k,k'<N_{A_i}$, 
$$c(a(k), b(k'))=j \text{ and } c(a(N_{A_i}+k), b(N_{A_i}+k'))=1-j.$$
\end{itemize}
The order is reverse inclusion.
\end{defn}
The following lemma  shows that under $\varphi_0(c)$,   forcing with $\mc{Q}_{\langle x^i_\alpha: i<n,\alpha<\omega_1\rangle}$ preserves $\varphi_0(c)$.
\begin{lem}\label{preservation}
Assume $\varphi_0(c)$.
If $0<n<\omega$, $A_0,..., A_{n-1}$ are in $\mc{C}^0$ and $\langle x^i_\alpha\in [A_i]^{<\omega}:\alpha<\omega_1\rangle$ is pairwise disjoint for each $i<n$, then $\mc{Q}_{\langle x^i_\alpha: i<n,\alpha<\omega_1\rangle}$ is c.c.c. and $\Vdash_{\mc{Q}_{\langle x^i_\alpha: i<n,\alpha<\omega_1\rangle}} \varphi_0(c)$.
\end{lem}
 \begin{proof}
 c.c.c. follows directly from $\varphi_0(c)$. Just note that the common part does not affect compatibility, i.e., for two conditions $p$ and $q$, $p$ is compatible with $q$ iff $p\setminus (p\cap q)$ is compatible with $q\setminus (p\cap q)$.\medskip
 
 Now suppose, towards a contradiction, that $\nVdash \varphi_0(c)$. Fix $p\in \mc{Q}_{\langle x^i_\alpha: i<n,\alpha<\omega_1\rangle}$ and $0<m<\omega$, $\mc{Q}_{\langle x^i_\alpha: i<n,\alpha<\omega_1\rangle}$-names $\dot{B}_0,...,\dot{B}_{m-1}$ in $\dot{\mc{C}}^0$ such that
$$p\Vdash \varphi_0(c, \dot{B}_0,...,\dot{B}_{m-1}) \text{ fails.}$$

Assume moreover that for $\mc{Q}_{\langle x^i_\alpha: i<n,\alpha<\omega_1\rangle}$-names $\dot{f}: \omega_1\ra S$ and $\dot{g}_i: \omega_1\ra [\dot{B}_i]^{<\omega}$ where $i<m$,
\begin{enumerate}
\item $p\Vdash \langle \dot{f}(\alpha): \alpha<\omega_1\rangle \text{ and } \langle \dot{g}_i(\alpha): i<m, \alpha<\omega_1\rangle$   witness the failure of $\varphi_0(c, \dot{B}_0,...,\dot{B}_{m-1})$.
\end{enumerate}

Extending $p$ if necessary,   find $t_0,..., t_{m-1}$ and $k_0,...,k_{m-1}$ such that
\begin{enumerate}\setcounter{enumi}{1}
\item $p\Vdash stem(\dot{B}_i)=t_i \text{ and } N_{\dot{B}_i}=k_i \text{ for any } i<m.$
\end{enumerate}

Now for any $\alpha<\omega_1$, find $p_\alpha<_{\mc{Q}_{\langle x^i_\alpha: i<n,\alpha<\omega_1\rangle}} p$  and $s_\alpha\in S$, $\langle y^i_\alpha: i<m \rangle$ such that 
\begin{enumerate}\setcounter{enumi}{2}
\item $p_\alpha\Vdash  \dot{f}(\alpha)=s_\alpha \text{ and } \dot{g}_i(\alpha)=y^i_\alpha \text{ for any } i<m.$
\end{enumerate}

Find $\Gamma\in [\omega_1]^{\omega_1}$ such that
\begin{enumerate}\setcounter{enumi}{3}
 \item $\{p_\alpha: \alpha\in \Gamma\}$ forms a $\Delta$-system with root $r$;
 \item $\langle ht(s_\alpha): \alpha\in \Gamma\rangle$ is strictly increasing and for any $\alpha<\beta$ in $\Gamma$, 
 $$ \max\{ht(s): s\in \bigcup\{ \bigcup y^i_\alpha: i<m\}\}< \min\{ht(s): s\in \bigcup\{ \bigcup y^i_\beta: i<m\}\}.$$
\end{enumerate}
Together with (3), we conclude that for any $i<m$, $\langle y^{i}_\alpha: \alpha\in \Gamma\rangle$ is pairwise disjoint.\medskip

 For $i<m$, let
$$B'_i=\{a: a\in y^i_\alpha \text{ for some } \alpha\in \Gamma\}.$$
It is straightforward to check, by (3) and (5), that $B'_i\in \mc{C}^0$   for any $i<m$. 

For $i<n$ and $\alpha\in \Gamma$, let
$$y^{m+i}_\alpha=\bigcup_{\gamma\in p_\alpha\setminus r} x^i_\gamma.$$
Note that by (4), for any $i<n$, $\langle y^{m+i}_\alpha: \alpha\in \Gamma\rangle$ is pairwise disjoint.

Applying $\varphi_0(c, B'_0,...,B'_{m-1}, A_0,...,A_{n-1})$ to $\langle s_\alpha: \alpha\in \Gamma\rangle$ and $\langle y^i_\alpha: i<m+n, \alpha\in \Gamma\rangle$, we find $\alpha<\beta$ in $\Gamma$  such that
\begin{enumerate}\setcounter{enumi}{5}
\item $s_\alpha<_S s_\beta$;
\item for any $i<m+n$, for any $a\in y^i_\alpha$ and $b\in y^i_\beta$, there is $j<2$ such that for any $k,k'< \frac{|a|}{2}$,
$$c(a(k), b(k'))=j \text{ and } c(a(\frac{|a|}{2}+k), b(\frac{|a|}{2}+k'))=1-j.$$
\end{enumerate}
(7) shows that $p_\alpha$ is compatible with $p_\beta$. Pick $p'<_{\mc{Q}_{\langle x^i_\alpha: i<n,\alpha<\omega_1\rangle}} p_\alpha, p_\beta$. But then (3), (6) and (7) show that
\begin{enumerate}\setcounter{enumi}{7}
\item $p'\Vdash  \varphi_0(c, \dot{B}_0,...,\dot{B}_{m-1})$ is true for the assignments $ \dot{f}, \langle \dot{g}_i: i<m\rangle$ witnessed by $\alpha<\beta.$
\end{enumerate}
This contradicts (1) and the fact that $p'<_{\mc{Q}_{\langle x^i_\alpha: i<n,\alpha<\omega_1\rangle}} p$.
 \end{proof}
 
 Now we are ready to prove the following consistency.
 \begin{thm}\label{varphi0}
It is consistent that there is a complete coherent Suslin tree $S\subset 2^{<\omega_1}$ and a coloring $c: [S]^2\ra 2$ such that $\mathrm{MA}_{\omega_1}(S)$ and $\varphi_0(c)$ hold.
 \end{thm}
 \begin{proof}
 By Fact \ref{fact1}, start with a model $V$ of GCH in which there is a complete coherent Suslin tree $S\subset 2^{<\omega_1}$ and a coloring $c: [S]^2\ra 2$ witnessing Pr$_0(S, 2, \omega)$. Then by Lemma \ref{initial}, $\varphi_0(c)$ holds in $V$.
 
 We then iterated force all  c.c.c. posets of size $\leq\omega_1$ that preserve $\varphi_0(c)$ with finite support in a standard way as the iteration to force Martin's Axiom (see \cite{ma} or \cite{jech}). Let $\mc{P}$ be the direct limit of the finite support iteration $\langle \mc{P}_\alpha, \dot{\mc{Q}}_\alpha: \alpha<\omega_2\rangle$.
 
 Let $G$ be a generic filter over $\mc{P}$. Then by Lemma \ref{iteration}, $V[G]\vDash \varphi_0(c)$. In particular, $S$ is still Suslin in $V[G]$. Now it suffices to prove that $V[G]\vDash $ MA$_{\omega_1}(S)$.
 
 Fix a c.c.c. poset $\mc{Q}$ such that $\Vdash_\mc{Q} S$ is Suslin and a collection of dense open sets $\{D_\alpha: \alpha<\omega_1\}$. Without loss of generality, assume $|\mc{Q}|\leq \omega_1$.\medskip
 
 First assumme that for some $\alpha<\omega_2$, $\mc{Q}\in V[G_\alpha]$ where $G_\alpha=\{p\up_\alpha: p\in G\}$ and 
 $$V[G_\alpha]\vDash ``\nVdash_{\mc{Q}} \varphi_0(c)".$$
 
 Now work in $V[G_\alpha]$. Repeating the argument in Lemma \ref{preservation}, we 
\begin{enumerate}
\item find $p\in \mc{Q}$ and $0<m<\omega$, $\mc{Q}$-names $\dot{B}_0,...,\dot{B}_{m-1}$ in $\dot{\mc{C}}^0$,  $\dot{f}: \omega_1\ra S$ and $\dot{g}_i: \omega_1\ra [\dot{B}_i]^{<\omega}$ where $i<m$ such that $p\Vdash_\mc{Q} \langle \dot{f}(\alpha): \alpha<\omega_1\rangle \text{ and } \langle \dot{g}_i(\alpha): i<m, \alpha<\omega_1\rangle$   witness the failure of $\varphi_0(c, \dot{B}_0,...,\dot{B}_{m-1})$;

\item  extending $p$ if necessary,  find $t_0,..., t_{m-1}$ and $k_0,...,k_{m-1}$ such that $p\Vdash stem(\dot{B}_i)=t_i \text{ and } N_{\dot{B}_i}=k_i \text{ for any } i<m$;
  
\item  and find, for each $\alpha<\omega_1$, $p_\alpha<_{\mc{Q}} p$  and $s_\alpha\in S$, $\langle y^i_\alpha: i<m \rangle$ such that $p_\alpha\Vdash  \dot{f}(\alpha)=s_\alpha \text{ and } \dot{g}_i(\alpha)=y^i_\alpha \text{ for any } i<m.$
\end{enumerate}

Now, find $\Gamma\in [\omega_1]^{\omega_1}$ such that
\begin{enumerate}\setcounter{enumi}{3}
 \item $\langle ht(s_\alpha): \alpha\in \Gamma\rangle$ is strictly increasing and for any $\alpha<\beta$ in $\Gamma$, 
 $$ \max\{ht(s): s\in \bigcup\{ \bigcup y^i_\alpha: i<m\}\}< \min\{ht(s): s\in \bigcup\{ \bigcup y^i_\beta: i<m\}\}.$$
\end{enumerate}
Together with (3), we conclude that for any $i<m$, $\langle y^{i}_\alpha: \alpha\in \Gamma\rangle$ is pairwise disjoint.\medskip

 For $i<m$, let
$$A_i=\{a: a\in y^i_\alpha \text{ for some } \alpha\in \Gamma\}.$$
It is straightforward to check, by (3) and (4), that $A_i\in \mc{C}^0$   for any $i<m$.  \medskip

Then (1)-(4) show that
\begin{enumerate}\setcounter{enumi}{4}
\item for any $\alpha<\beta$ in $\Gamma$, if $p_\alpha$ is compatible with $p_\beta$, then either $s_\alpha$ is incomparable with $s_\beta$ or there are $i<m$,   $a\in y^i_\alpha$ and  $b\in y^i_\beta$ such that no $j<2$ satisfies for any $k,k'<k_i$, 
$$c(a(k), b(k'))=j \text{ and } c(a(k_i+k), b(k_i+k'))=1-j.$$
\end{enumerate}

By Lemma \ref{preservation}, $\mc{Q}_{\langle y^i_\alpha: i<m,\alpha\in \Gamma\rangle}$ is a c.c.c. poset of size $\leq \omega_1$ that preserves $\varphi_0(c)$ in $V[G_\beta]$ for any $\beta\in [\alpha,\omega_2)$. Omitting a countable part, we may assume that
$$\Vdash_{\mc{Q}_{\langle y^i_\alpha: i<m,\alpha\in \Gamma\rangle}} \dot{H} \text{ is uncountable}$$
where $\dot{H}$ is the canonical name for the $\mc{Q}_{\langle y^i_\alpha: i<m,\alpha\in \Gamma\rangle}$-generic filter.

 We may assume that for some $\eta>\alpha$, $\mc{Q}_{\langle y^i_\alpha: i<m,\alpha\in \Gamma\rangle}$ is $\dot{\mc{Q}}_\eta^{G_\eta}$.

Now, for $G(\eta)=\{p(\eta): p\in G\}$, $\bigcup G(\eta)$  is an uncountable subset of $\Gamma$ such that
\begin{enumerate}\setcounter{enumi}{5}
\item  for any $\alpha<\beta$ in $\bigcup G(\eta)$ and $i<m$, for any  $a\in y^i_\alpha$ and  $b\in y^i_\beta$,  there is $j<2$ such that for any $k,k'<k_i$, 
$$c(a(k), b(k'))=j \text{ and } c(a(k_i+k), b(k_i+k'))=1-j.$$
\end{enumerate}

Now work in $V[G]$.
(5) and (6) show that $\{(p_\alpha, s_\alpha): \alpha\in \bigcup G(\eta)\}$ is an uncountable antichain of $\mc{Q}\times S$.  This contradicts our assumption that $\Vdash_\mc{Q} S$ is  Suslin.  This contradiction shows that for any $\alpha<\omega_2$, if $\mc{Q}\in V[G_\alpha]$, then
 $$V[G_\alpha]\vDash ``\Vdash_{\mc{Q}} \varphi_0(c)".$$
 
 But then for some $\eta<\omega_2$ , $\mc{Q}$ is isomorphic to $\dot{\mc{Q}}_\eta^{G_\eta}$   and $\{D_\alpha: \alpha<\omega_1\}\in V[G_\eta]$.  So there is a filter meeting them all. This shows that $V[G]\vDash $ MA$_{\omega_1}(S)$ and finishes the proof of the theorem.
 \end{proof}
Now Theorem \ref{thm1} follows from Corollary \ref{cor} and Theorem \ref{varphi0}.
%\begin{cor} $\mathrm{MA}_{\omega_1}(S)[S]$ is consistent with $\mathrm{Pr}_1(\omega_1, 2, \omega)$. In particular, $\mathrm{MA}_{\omega_1}(S)[S]$ does not imply $\mc{K}_2$. \end{cor}

The proof of Theorem \ref{varphi0} suggests a new forcing axiom.
\begin{defn}
For a complete coherent Suslin tree $S\subset 2^{<\omega_1}$ and a coloring $c: [S]^2\ra 2$, $\mathrm{MA}_{\omega_1}(S, c, \varphi_0)$ is the assertion that
\begin{enumerate}[(i)]
\item $\varphi_0(c)$ holds;
\item if $\mc{P}$ is a c.c.c. forcing that preserves $\varphi_0(c)$ and $\{D_\alpha: \alpha<\omega_1\}$ is a collection of dense subsets of $\mc{P}$, then there is a filter meeting them all.
\end{enumerate}
\end{defn}
Similar as  MA$_{\omega_1}(S)[S]$, we say that MA$_{\omega_1}(S, c, \varphi_0)[S]$  holds if  the universe is a forcing extension by $S$ over a model of MA$_{\omega_1}(S, c, \varphi_0)$.

\begin{prop}\label{prop}
For a complete coherent Suslin tree $S\subset 2^{<\omega_1}$ and a coloring $c: [S]^2\ra 2$, the following are equivalent.
\begin{enumerate}
\item $\mathrm{MA}_{\omega_1}(S, c, \varphi_0)$.
\item $\mathrm{MA}_{\omega_1}(S)$ and $\varphi_0(c)$.
\end{enumerate}
\end{prop}
\begin{proof}
(1)$\Rightarrow$(2). We only need to prove MA$_{\omega_1}(S)$. Fix a c.c.c. poset $\mc{P}$ that preserves $S$. It suffices to prove that $\Vdash_\mc{P} \varphi_0(c)$. 

Suppose otherwise. The argument in Lemma \ref{preservation} or Theorem \ref{varphi0} shows that for some $0<m<\omega$, $A_0,...,A_{m-1}$ in $\mc{C}^0$ and pairwise distinct $\langle s_\alpha: \alpha<\omega_1\rangle$, pairwise disjoint $\langle y^i_\alpha\in [A_i]^{<\omega}: \alpha<\omega_1\rangle $ for each $i<m$, $\mc{Q}_{\langle y^i_\alpha: i<m,\alpha<\omega_1\rangle}$ forces an uncountable antichain of $\mc{P}\times S$.

But by Lemma \ref{preservation}, $\mc{Q}_{\langle y^i_\alpha: i<m,\alpha<\omega_1\rangle}$ is c.c.c. and preserves $\varphi_0(c)$. So applying MA$_{\omega_1}(S, c, \varphi_0)$, we get an uncountable antichain of $\mc{P}\times S$. This contradicts our assumption that $\mc{P}$ is c.c.c. and preserves   $S$.\medskip

(2)$\Rightarrow$(1). Since  $\varphi_0(c)$ implies that $S$ is Suslin, any $\omega_1$-preserving forcing that preserves $\varphi_0(c)$ preserves $S$.
\end{proof}
\textbf{Remark.} Although MA$_{\omega_1}(S, c, \varphi_0)$ is not a new property, it may make a difference if we change the forcing axiom. For example, if we may get PFA$(S, c, \varphi_0)$, %(e.g., using Woodin's $\mathbb{P}_{\max}$ forcing method), 
then the new forcing axiom might be totally different from PFA$(S)+\varphi_0(c)$. %We shall not go into details since it is out of the scope of the paper.\medskip

Together with Theorem \ref{varphi0} and Corollary \ref{cor}, we get the following.
\begin{cor}
\begin{enumerate}
\item $\mathrm{MA}_{\omega_1}(S, c, \varphi_0)$ is consistent.
\item $\mathrm{MA}_{\omega_1}(S, c, \varphi_0)[S]$ implies $\mathrm{MA}_{\omega_1}(S)[S]$ and $\mathrm{Pr}_1(\omega_1, 2, \omega)$.
\end{enumerate}
\end{cor}

\textbf{Remark.} Assume MA$_{\omega_1}(S, c, \varphi_0)$. Then for any $\mc{A}=\{(s_\alpha, a_\alpha)\in S\times [\omega_1]^n: \alpha<\omega_1\}$, viewed as an $S$-name of an uncountable pairwise disjoint subset of $[\omega_1]^n$, the proof of Lemma \ref{lem1} shows that  there are $s\in S$, $f: \omega_1\ra \omega_1$ and $g: \omega_1\ra \omega_1$ such that $D_{s_{f(\alpha)}, s_{g(\alpha)}}=\{ht(s)\}$ and  $B=\{b_\alpha: \alpha<\omega_1\}\in \mc{C}^0_n$ where for any $i<n$, 
\begin{itemize} 
\item $b_\alpha(i)=s_{f(\alpha)}\up_{a_{f(\alpha)}(i)}$ and $b_\alpha(n+i)=s_{g(\alpha)}\up_{a_{g(\alpha)}(i)}$.
\end{itemize}
Applying MA$_{\omega_1}(S, c, \varphi_0)$ to $\mc{Q}_{\langle \{b_\alpha\}: \alpha<\omega_1\rangle}$, we get   $\Gamma\in [\omega_1]^{\omega_1}$ such that 
\begin{itemize}
\item for any $\alpha<\beta$ in $\Gamma$, there is $j<2$ such that for any $k,k'<n$,
$$c(b_\alpha(k), b_\beta(k'))=j \text{ and } c(b_\alpha(n+k), b_\beta(n+k'))=1-j.$$
\end{itemize}
Then $\mc{B}=\{(s_{f(\alpha)}, a_{f(\alpha)}): \alpha\in \Gamma\}\cup \{(s_{g(\alpha)}, a_{g(\alpha)}): \alpha\in \Gamma\}$ is an uncountable subset of $\mc{A}$ that is ``half'' 0-homogeneous (and ``half'' 1-homogeneous). I.e., for any $\alpha<\beta$ in $\Gamma$, either
\begin{itemize}
\item $c[\{s_{f(\alpha)}\up_{a_{f(\alpha)}(i)}: i<n\} \times \{s_{f(\beta)}\up_{a_{f(\beta)}(i)}: i<n\}]=\{0\}$ or
\item  $c[\{s_{g(\alpha)}\up_{a_{g(\alpha)}(i)}: i<n\} \times \{s_{g(\beta)}\up_{a_{g(\beta)}(i)}: i<n\}]=\{0\}$.
\end{itemize}
In other words, for each pair $\alpha<\beta$ in $\Gamma$, one of the two choices --- $(f(\alpha), f(\beta))$ and $(g(\alpha), g(\beta))$ --- gives a 0-homogeneous pair.
So our strategy of constructing MA$_{\omega_1}(S)[S]$ together with Pr$_1(\omega_1, 2, \omega)$ can be viewed in this way: 
\begin{itemize}
\item we force an uncountable ``half'' 0 (and 1) homogeneous set for every possible $S$-name of an uncountable pairwise disjoint family;
\item in the iteration, we force every small c.c.c. poset that does not destroy the possibility of   adding ``half'' 0 (and 1) homogeneous sets.
\end{itemize}\medskip

The key is that ``half'' homogeneity is strong enough to be preserved in the iteration process but not too strong to add an uncountable homogeneous set. In the following section, we will apply this idea to investigate stronger colorings.\medskip

However, this idea may not work for the consistency of PFA$(S)[S]+\neg\mc{K}_2$. In the rest of the section, we will prove that $\varphi_0(c)$ is rejected by PFA$(S)$.

We will use the following notation for convenience.
   \begin{defn}
   For $\alpha<\omega_1$, $\Psi_\alpha: 2^{(\alpha, \omega_1)} \ra 2^{(\alpha, \omega_1)}$ where $2^{(\alpha, \omega_1)}=\bigcup_{\beta\in (\alpha, \omega_1)} 2^\beta$ is defined by for any $x\in 2^{(\alpha, \omega_1)}$, 
   $$\Psi_\alpha(x)=x\up_\alpha\sm (1-x(\alpha))\sm x\up_{(\alpha, ht(x))},$$
    i.e.,     $\Psi_\alpha(x)$ is the node in $2^{ht(x)}$ such that $D_{x, \Psi_\alpha(x)}=\{\alpha\}$.
      \end{defn}

For the rest of this section, for $a\in [S]^{<\omega}$, denote 
$$\Delta(a)=\min\{\Delta(s,t): s,t\in a\}\text{ and }$$
$$\wedge (a)=a(0)\up_{\Delta(a)}.$$ 
Note that   $\Delta(a)\leq ht(s)$ for any $s\in a$ since $\Delta(s,s)=ht(s)$. $\wedge (a)$ is the maximal common part of $a$.

We will need  the following property.
\begin{lem}\label{splitting tree}
Assume $\varphi_0(c)$. Suppose that 
\begin{enumerate}[$(i)$]
\item $0<n<\omega$, $A\subset [S]^2$ and $x\in S$;
\item $T\subset S^{\leq n}$ such that $\emptyset\in T$ and for any $\sigma\in T\cap S^{<n}$, $\{t\in S: \sigma^\smallfrown t\in T\}$ is uncountable;
\item for any $\sigma\in T\cap S^n$,  $\Delta(\{\sigma(i): i<n\})\geq otp(ht(\sigma(0))\cap\{ht(s): \langle s\rangle \in T\} )$;\footnote{For $\Gamma\subset Ord$, $otp(\Gamma)$ is the order type of $\Gamma$.}
\item   $\{t: \langle t\rangle\in T\}$ is dense above $x$;
\item  $a(0)<_S a(1)$ whenever $a\in A$ and $\{a(0): a\in A\}$ is dense above $x$.
\end{enumerate}
Then there are $a\in A$ and $\sigma\in T\cap S^n$ such that $a(1)<_S \sigma(0)$ and $c(a(0), \sigma(i))=0$ for any $i<n$.
\end{lem}
\begin{proof}
%Going to a subtree of $T$ and extending $x$ if necessary, we may assume that for any $\sigma\in T$, for any $i<|\sigma|$, $\sigma(i)>_S x$.

Let $B$ be the collection of all $b\in [S]^6$ such that
\begin{enumerate}
\item  $x^\smallfrown 0<_S b(0)<_S b(1)<_S b(2)$ and $x^\smallfrown 1<_S b(3)<_S b(4)<_S b(5)$;
\item $ht(b(2))< ht(b(3))$ and $D_{b(2), b(3)}=\{ht(x)\}$;
\item $\{b(0), b(1)\}\in A$ and $\{b(3), b(4)\}\in A$.
\end{enumerate}
Fix $f: T\cap S^{<n}\ra S$ such that for any $\sigma\in T\cap S^{<n}$, $\{t\in S: \sigma^\smallfrown t\in T\}$ is dense above $f(\sigma)$. Note that by $(iii)$, for $\sigma\in T\cap S^{<n}\setminus \{\emptyset\}$, 
$$f(\sigma)\geq_S \sigma(0)\up_{otp(ht(\sigma(0))\cap\{ht(s): \langle s\rangle \in T\})}.$$

For each $\alpha<\omega_1$, we will choose   $\langle b^\alpha_\sigma\in B: \sigma\in n^{\leq n}\setminus\{\emptyset\}\rangle$ such that 
\begin{enumerate}\setcounter{enumi}{3}
\item for any $\sigma<_{ho} \tau$ in $n^{\leq n}\setminus\{\emptyset\}$, $ht(b^\alpha_\sigma(5))<ht(b^\alpha_\tau(0))$;
\item for any $\sigma\in n^{< n}$ and $j< n$, $\langle b^\alpha_{\sigma^\smallfrown i}(2): i\leq j \rangle\in T$;
\item for any $\sigma\in n^{\leq n}\setminus\{\emptyset\}$, $\langle b^\alpha_{\sigma\up_i}(5): 0<i\leq |\sigma|\rangle\in T$;
\item $\Delta(\{b^\alpha_\sigma(0): \sigma\in n^{\leq n}\setminus\{\emptyset\}\})\geq \alpha$.
\end{enumerate}
Fix $\alpha$ and we choose $b^\alpha_\sigma$s' by induction on $<_{ho}$. First choose $x_\alpha$ extending $x^\smallfrown 0$ of height $\geq\alpha$. Let 
$$y_\alpha=\Psi_{ht(x)}(x_\alpha).$$
To get (7), it suffices to choose $b^\alpha_\sigma$s' such that 
\begin{enumerate}\setcounter{enumi}{7}
\item $b^\alpha_\sigma(0)>_S x_\alpha$ whenever $\sigma\in n^{\leq n}\setminus\{\emptyset\}$.
 \end{enumerate}

 Suppose that $\sigma\in n^{\leq n}\setminus\{\emptyset\}$ and $b^\alpha_\tau$ has been chosen for $\emptyset<_{ho} \tau<_{ho} \sigma$. The general procedure of choosing $b^\alpha_\sigma$ is the following:
 \begin{itemize}
 \item we first choose appropriate $u_\sigma$ and $ v_\sigma$ such that $v_\sigma=\Psi_{ht(x)}(u_\sigma)\geq_S y_\alpha$;
 \item then arbitrarily choose   $b^\alpha_\sigma(0)<_S b^\alpha_ \sigma(1)$ such that $u_ \sigma <_S b^\alpha_ \sigma(0)$, $\{b^\alpha_ \sigma(0), b^\alpha_ \sigma(1)\}\in A$ and (4) holds;
 \item then choose an appropriate $b^\alpha_\sigma(2)$;
 \item then arbitrarily choose $b^\alpha_ \sigma(3)<_S b^\alpha_ \sigma(4)$ such that $\Psi_{ht(x)}(b^\alpha_ \sigma(2))<_S b^\alpha_ \sigma(3)$ and $\{b^\alpha_ \sigma(3), b^\alpha_ \sigma(4)\} \in A$;
 \item finally choose an appropriate $b^\alpha_\sigma(5)$.
 \end{itemize}
 Note that the choice of $b^\alpha(0),b^\alpha(1),b^\alpha(3)$ and $b^\alpha(4)$ is trivial by $(v)$. So we only  describe the method of choosing $u_\sigma$ (or $v_\sigma$), $b^\alpha_\sigma(2)$ and $b^\alpha_\sigma(5)$ by cases.
 
 \begin{case}
 $\sigma=\langle 0\rangle$.
 \end{case}
Let $u_{\langle 0\rangle}= x_\alpha$.

 $b^\alpha_{\langle 0\rangle}(2)$ is chosen to be any extension of $b^\alpha_{\langle 0\rangle}(1)$ such that $\langle b^\alpha_{\langle 0\rangle}(2)\rangle \in T$ and 
\begin{enumerate}
\item[(9)$_{\langle 0\rangle}$] $otp(ht(b^\alpha_{\langle 0\rangle}(2))\cap \{ht(s): \langle s\rangle \in T\})>ht(u_{\langle 0\rangle})$.
\end{enumerate}

$b^\alpha_{\langle 0\rangle}(5)$ is any extension of $b^\alpha_{\langle 0\rangle}(4)$ such that $\langle b^\alpha_{\langle 0\rangle}(5)\rangle\in T$. 

It is straightforward to check (1)-(3). So $b^\alpha_{\langle 0\rangle}\in B$. Also, (4)-(6) and (8) follow from our construction.

 \begin{case}
 $\sigma=\langle i\rangle$ for some $0<i<n$. 
 \end{case}
First note by induction hypothesis (5), 
 $$\langle b^\alpha_{{\langle j\rangle}}(2): j< i \rangle\in T.$$
 Let $u_{\langle i\rangle}=f( \langle b^\alpha_{{\langle j\rangle}}(2): j< i \rangle)$. By $(iii)$, (9)$_{\langle 0\rangle}$ and the definition of $f$, 
 $$u_{\langle i\rangle}>_S b^\alpha_{\langle 0\rangle}(2)\up_{ht(u_{\langle 0\rangle})}=u_{\langle 0\rangle}=x_\alpha.$$

Then, by definition of $u_{\langle i\rangle}$ and the fact that $u_{\langle i\rangle}<_S b^\alpha_{\langle i\rangle}(0)<_S b^\alpha_{\langle i\rangle}(1)$, choose $b^\alpha_{\langle i\rangle}(2)$ to be any extension of $b^\alpha_{\langle i\rangle}(1)$ such that $\langle b^\alpha_{\langle j\rangle}(2): j< i \rangle^\smallfrown b^\alpha_{\langle i\rangle}(2)  \in T$ and
\begin{enumerate}
\item[(9)$_{\langle i\rangle}$] $otp(ht(b^\alpha_{\langle i\rangle}(2))\cap \{ht(s): \langle s\rangle \in T\})>ht(u_{\langle i\rangle})$.
\end{enumerate}

Finally choose $b^\alpha_{\langle i\rangle}(5)$ to be any extension of $b^\alpha_{\langle i\rangle}(4)$ such that $\langle  b^\alpha_{\langle i\rangle}(5)\rangle\in T$. 

It is straightforward to check that (1)-(6) and (8)  hold.

\begin{case}
$\sigma=\tau^\smallfrown 0$ for some $\tau\neq \emptyset$.
\end{case}
By induction hyposthesis (6), 
$$\langle b^\alpha_{\tau\up_i}(5): 0<i\leq |\tau|\rangle\in T.$$
Let $v_\sigma=f(\langle b^\alpha_{\tau\up_i}(5): 0<i\leq |\tau|\rangle)$ and $u_\sigma=\Psi_{ht(x)}(v_\sigma)$.

Note by our choice of $b^\alpha_{\langle \tau(0)\rangle}$ and  (9)$_{\langle \tau(0)\rangle}$,
$$otp(ht(b^\alpha_{\langle \tau(0)\rangle}(5))\cap \{ht(s): \langle s\rangle \in T\})>ht(u_{\langle \tau(0)\rangle}).$$
Together with $(iii)$, 
$$v_\sigma>_S b^\alpha_{\langle \tau(0)\rangle}(5)\up_{ht(u_{\langle \tau(0)\rangle})}=v_{\langle \tau(0)\rangle} \geq_S y_\alpha.$$

Then choose $b^\alpha_\sigma(2)$ to be any extension of $b^\alpha_\sigma(1)$ such that $\langle b^\alpha_\sigma(2)\rangle \in T$ and 
\begin{enumerate}\setcounter{enumi}{8}
\item[(9)$_{\tau\sm 0}$] $otp(ht(b^\alpha_\sigma(2))\cap \{ht(s): \langle s\rangle \in T\})>ht(u_\sigma)$.
\end{enumerate}

Now note that $b^\alpha_\sigma(4)$ extends $v_\sigma$. Hence we can choose $b^\alpha_\sigma(5)$ to be any extension of  $b^\alpha_\sigma(4)$ such that $\langle b^\alpha_{\tau\up_i}(5): 0<i\leq |\tau|\rangle^\smallfrown b^\alpha_\sigma(5)\in T$. 

It is straightforward to check that (1)-(6) and (8)  hold. 

\begin{case}
$\sigma=\tau^\smallfrown i$ for some $\tau\neq \emptyset$ and $0<i<n$.
\end{case}
 By induction hypothesis (5), 
 $$\langle b^\alpha_{\tau\sm j}(2): j< i \rangle\in T.$$
 Let $u_\sigma=f( \langle b^\alpha_{\tau^\smallfrown j}(2): j< i \rangle)$.

By $(iii)$, (9)$_{\tau\sm 0}$ and our choice of $b^\alpha_{\tau\sm 0}$, 
$$u_\sigma>_S b^\alpha_{\tau\sm 0}(2)\up_{ht(u_{\tau^\smallfrown 0})}=u_{\tau^\smallfrown 0}\geq_S x_\alpha.$$
Hence $v_\sigma>_S v_{\tau^\smallfrown 0}$.

Then, by our choice of $u_\sigma$, choose $b^\alpha_\sigma(2)$ to be any extension of $b^\alpha_\sigma(1)$ such that $\langle b^\alpha_{\tau^\smallfrown j}(2): j< i \rangle^\smallfrown b^\alpha_\sigma(2)  \in T$.

Note that $b^\alpha_\sigma(4)$ extends $v_\sigma>_S v_{\tau^\smallfrown 0}$. By definition of $v_{\tau^\smallfrown 0}$ in Case 3, we can choose $b^\alpha_\sigma(5)$ to be any extension of  $b^\alpha_\sigma(4)$ such that $\langle b^\alpha_{\tau\up_i}(5): 0<i\leq |\tau|\rangle^\smallfrown b^\alpha_\sigma(5)\in T$. 

It is straightforward to check that (1)-(6) and (8)  hold.\bigskip

This finishes the   construction at $\alpha$.

Recall that $b^\alpha_\sigma\cap S\up_\alpha=\emptyset$ for any $\alpha$ and $\sigma$. So we can find $\Gamma\in [\omega_1]^{\omega_1}$ such that $ht(b^\alpha_\sigma(5))<\beta\leq ht(b^\beta_\tau(0))$ whenever $\alpha<\beta$ are in $\Gamma$.

Together with (1)-(4),  $D=\{b^\alpha_\sigma: \alpha\in \Gamma, \sigma\in n^{\leq n}\setminus\{\emptyset\}\}$ is in $\mc{C}^0$. Applying $\varphi_0(c, D)$ to $\langle b^\alpha_{\langle 0\rangle}(4): \alpha\in \Gamma\rangle$ and $\langle \{b^\alpha_\sigma: \sigma\in n^{\leq n}\setminus\{\emptyset\}\}: \alpha\in \Gamma\rangle$, we get $\alpha<\beta$ in $\Gamma$ such that
\begin{enumerate}\setcounter{enumi}{9}
\item $b^\alpha_{\langle 0\rangle}(4)<_S b^\beta_{\langle 0\rangle}(4)$;
\item for any $\tau$ in $n^{\leq n}\setminus\{\emptyset\}$, there is $j<2$ such that
$$c(b^\alpha_{\langle 0\rangle}(0), b^\beta_\tau(2))=j \text{ and } c(b^\alpha_{\langle 0\rangle}(3), b^\beta_\tau(5))=1-j.$$
\end{enumerate}

Now by (7), (10), the definition of $\mc{C}^0$  and our choice of $\Gamma$, for any $\tau$ in $n^{\leq n}\setminus\{\emptyset\}$,  
$$b^\alpha_{\langle 0\rangle}(1)<_S b^\beta_\tau(2)\text{ and }  b^\alpha_{\langle 0\rangle}(4)<_S b^\beta_\tau(5).$$

First assume that for some $\tau\in n^{< n}$, 
$$c(b^\alpha_{\langle 0\rangle}(0), b^\beta_{\tau\sm i}(2))=0\text{ for all }i<n.$$
 Then the conclusion of the lemma holds since  by (3), $\{b^\alpha_{\langle 0\rangle}(0), b^\alpha_{\langle 0\rangle}(1)\}\in A$ and by (5) $\langle b^\beta_{\tau\sm i}(2): i<n\rangle \in T$.

Now suppose that for any $\tau\in n^{< n}$, there is $i<n$ such that $c(b^\alpha_{\langle 0\rangle}(0), b^\beta_{\tau\sm i}(2))=1$. Then, by an induction on the length, choose  $\sigma\in n^n$ such that $c(b^\alpha_{\langle 0\rangle}(0), b^\beta_{\sigma\up_i}(2))=1$ for any $0<i\leq n$. By (11), 
$$c(b^\alpha_{\langle 0\rangle}(3), b^\beta_{\sigma\up_i}(5))=0\text{ for any } 0<i\leq n.$$
 Then the conclusion of the lemma holds since by (3), $\{b^\alpha_{\langle 0\rangle}(3), b^\alpha_{\langle 0\rangle}(4)\}\in A$ and by (6), $\langle b^\beta_{\sigma\up_i}(5): 0<i\leq n\rangle \in T$.
\end{proof}

\begin{lem}\label{modelcomb}
Assume $\varphi_0(c)$. Suppose that
\begin{enumerate}[$(i)$]
\item $\mc{M}_0\in \mc{M}_1\in...\in\mc{M}_{n-1}$ is a $\in$-chain of countable elementary submodels of $H(\omega_2)$ containing $S, c$ where $0<n<\omega$;
\item $b\in [S]^n$ such that 
\begin{itemize}
\item for any $i<n-1$, $b(i)\in \mc{M}_{i+1}\setminus\mc{M}_i$ and $b_{n-1}\notin \mc{M}_{n-1}$,
\item $\Delta(b)\geq \mc{M}_0\cap \omega_1$;
\end{itemize}
\item $A\subset [S]^2$ is in $\mc{M}_0$ such that
\begin{itemize}
\item for any $a\in A$, $a(0)<_S a(1)$;
\item $\{a(0): a\in A\}$ is dense above $\wedge(b)\up_\xi$ for some $\xi\in \mc{M}_0\cap \omega_1$.
\end{itemize}
\end{enumerate}
Then there is $a\in A\cap \mc{M}_0$ such that $a(1)<_S \wedge(b)$ and $c(a(0), b(i))=0$ for any $i<n$.
\end{lem}

 \begin{proof}
 Suppose otherwise. Let $B$ collect all $\langle\{\delta, v_0\}, v_1,..., v_{n-1}\rangle$ such that 
 \begin{itemize}
 \item $\delta\in [\xi, \omega_1)$ and $v_i\in S$ for any $i<n$;
 \item  $\Delta(\{v_i: i<n\})\geq \delta$ and $v_0\up_\xi=b(0)\up_\xi$;

 \item there is no $a\in A$ such that $a(1)<_S v_0\up_\delta$ and $c(a(0), v_i)=0$ for any $i<n$.
 \end{itemize}
 Then $B\in \mc{M}_0$ and $\langle\{\mc{M}_0\cap \omega_1, b(0)\}, b(1),...,b(n-1)\rangle\in B$.
 
 Let $T$ be the downward closure of $B$.
 
 Find, in $\mc{M}_0$, a downward closed subtree $T^*$ of $T$ such that
 \begin{itemize}
 \item $\emptyset\in T^*$ and $\langle\{\mc{M}_0\cap \omega_1, b(0)\}, b(1),...,b(n-1)\rangle\in T^*$;
 \item every $\sigma\in T^*$ of height $<n$ has uncountably many immediate successors.
 \end{itemize}
 
  Since $\langle\{\mc{M}_0\cap \omega_1, b(0)\}\rangle\in T^*$ and $T^*\in \mc{M}_0$, $\{u\up_\delta: \langle \{\delta,u\}\rangle\in T^*\}$ is dense above $b(0)\up_\eta$ for some $\eta\in [\xi, \mc{M}_0\cap \omega_1)$. Let
 $$x=b(0)\up_\eta.$$
 
%Find a subtree $T'\subset T^*$ such that $\emptyset\in T'$ and
% \begin{itemize}  \item for any $\langle(\eta, u)\rangle\neq \langle(\delta, v)\rangle$ in $T^*$, either $ht(u)<\delta$ or $ht(v)<\eta$;   \item every $\sigma\in T^*$ of height $<n$ has uncountably many immediate successor;   \item $\{u: \langle(\eta, u)\rangle\in T'\text{ for some } \eta\}$ is dense above $x$.  \end{itemize}
 
 Let $\{s_\alpha: \alpha<\omega_1\}$ enumerate $\{s\in S: s\geq_S x\}$. Now choose, by induction on $\alpha$, $a_\alpha\in A$ and $\langle \{\delta_\alpha, u_\alpha\}\rangle\in T^*$ such that
 \begin{enumerate}
 \item $s_\alpha<_S a_\alpha(0)$ and $s_\alpha<_S u_\alpha\up_{\delta_\alpha}$;
 \item $\sup\{ht(u_\beta): \beta<\alpha\}< ht(a_\alpha(0))<ht(a_\alpha(1))< \delta_\alpha$.
 \end{enumerate}
 
 Let $T'=\{\emptyset\}\cup \{\langle u_\alpha, v_1,..., v_{i}\rangle: \langle \{\delta_\alpha, u_\alpha\}, v_1,..., v_i\rangle \in T^*\}$.
 
 Let $A'=\{a_\alpha: \alpha<\omega_1\}$.
 
 It is straightforward to check that $n, A', x, T'$ satisfy $(i)-(v)$ in Lemma \ref{splitting tree}. Then by Lemma \ref{splitting tree}, there are $a\in A'$ and $\sigma\in T'\cap S^n$ such that $a(1)<_S \sigma(0)$ and $c(a(0), \sigma(i))=0$ for any $i<n$.
 
 Assume $a=a_\alpha$ and $\sigma(0)=u_\beta$ for some $\alpha, \beta$. By (2), $\alpha\leq \beta$ and hence $ht(a(1))<\delta_\beta$. So $a(1)<_S \sigma(0)\up_{\delta_\beta}$ and $c(a(0), \sigma(i))=0$ for any $i<n$. But this contradicts the definition of $B$ since $\langle\{\delta_\beta, \sigma(0)\}, \sigma(1),..., \sigma(n-1)\rangle \in B$.
 \end{proof}

 \begin{thm}\label{PFA rejects 0}
 Suppose that $S$ is a complete coherent Suslin tree and $c: [S]^2\ra 2$ is a coloring.
 Then $\mathrm{PFA}(S)$ implies that $\varphi_0(c)$ fails.
 \end{thm}
 \begin{proof}
 Suppose, towards a contradiction, that both PFA$(S)$ and $\varphi_0(c)$ hold. 
 
 Let $\mc{P}$ be the poset consisting of $p=(\mathscr{N}_p, F_p)$ such that
 \begin{enumerate}[$(i)$]
 \item $\ms{N}_p$ is a finite $\in$-chain of countable elementary subodels of $H(\omega_2)$ containing everything relevant;
 \item $F_p\in [S]^{<\omega}$ is separated by $\ms{N}_p$, i.e., for any $s\neq t$ in $F_p$, there is $\mc{N}\in \ms{N}_p$ such that $|\{s,t\}\cap \mc{N}|=1$;
 \item for any $s<_S t$ in $F_p$, $c(s,t)=0$.
 \end{enumerate}
 The order is coordinate-wise extension.
 
 We first prove that $\mc{P}$ is proper and $\Vdash_\mc{P} S$ is Suslin. Fix a countable elementary submodel $\mc{M}\prec H(\kappa)$ containing everything relevant where $\kappa$ is a large enough regular cardinal. It suffices to prove that $(r,s)$ is $(\mc{M}, \mc{P}\times S)$-generic whenever $\mc{M}\cap H(\omega_2)\in \ms{N}_r$.
 
  Now fix a dense open $D\subset \mc{P}\times S$ in $\mc{M}$. We will find a condition in $D\cap\mc{M}$ compatible with $(r,s)$. Extending $(r,s)$, we may assume that $(r,s)\in D$.
 
 Without loss of generality,   assume $\ms{N}_r\cap \mc{M}=\emptyset$ and $F_r\cap \mc{M}=\emptyset$. Extending $s$ if necessary, we may assume that $ht(s)>\max\{ht(t): t\in F_r\}$. Let 
 $$n=|F_r|.$$
  By coherence of $S$, choose $\delta\in \mc{M}\cap \omega_1$ such that 
 \begin{enumerate}
\item  for any $t\in F_r$, $s\up_{[\delta, \mc{M}\cap \omega_1)}=t\up_{[\delta, \mc{M}\cap \omega_1)}$.
 \end{enumerate}
 Let $$I=\{i<n: F_r(i)\up_{[\delta, ht(F_r(i)))}=s\up_{[\delta, ht(F_r(i)))}\}.$$
In particular,  for any $i<j$ in $I$, $F_r(i)\up_{[\delta, ht(F_r(i)))}=F_r(j)\up_{[\delta, ht(F_r(i)))}$.

  Now, by a standard reverse induction, we find a downward closed tree $T\subset S^{\leq n+1}$  in $\mc{M}$ such that
  \begin{enumerate}\setcounter{enumi}{1}
  \item $\emptyset\in T$, $\langle F_r(i): i<n\rangle^\smallfrown s \in T$;
  \item for any $a\in T\cap S^{<n}$, there are uncountably many $t\in S$ such that $a^\smallfrown t\in T$;
  \item every $a\in T\cap S^n$ has   an extension in $T$;
  \item for any $a\in T\cap S^{n+1}$, there is a condition $(p,t)\in D$ such that 
  \begin{enumerate}[(5.1)]
  \item $F_p=\{a(i): i<n\}$,  $t=a(n)$ and $ht(a(0))<...<ht(a(n))$;
  \item for any $u\in F_p$, $t\up_{[\delta, \mc{M}_p\cap \omega_1)}=u\up_{[\delta, \mc{M}_p\cap \omega_1)}$ where $\mc{M}_p$ is the $\in$-least element of $\ms{N}_p$;
  \item for any $i<n$, $a(i)\up_\delta=F_r(i)\up_\delta$ and $a(n)\up_\delta=s\up_\delta$;
  \item   $I=\{i<n: a(i)\up_{[\delta, ht(a(i)))}=t\up_{[\delta, ht(a(i)))}\}$.
  \end{enumerate}
  \end{enumerate}
  
    Let $I_i$ be the $i$th element of $I$ in the increasing enumeration. 
  By   $(3)$, for each $i<|I|-1$ and $a\in T\cap S^{I_i+1}$, choose $f(a)>_S a(I_i)$ such that
  \begin{enumerate}\setcounter{enumi}{5}
  \item $\{a(I_i)\up_\delta^\smallfrown b(I_{i+1})\up_{[\delta, ht(b(I_{i+1})))}: b\in T\cap S^{I_{i+1}+1}$ extending $a\}$ is dense above $f(a)$.
   \end{enumerate}
   To see that $f(a)>_S a(I_i)$, note by (5.4), for any $b\in T\cap S^{\geq I_{i+1}+1}$ extending $a$, 
   $$a(I_i)\up_{[\delta, ht(a(I_i)))}=b(I_{i+1})\up_{[\delta, ht(a(I_i)))}\text{ and hence }$$
   $$a(I_i)\up_\delta^\smallfrown b(I_{i+1})\up_{[\delta, ht(b(I_{i+1})))}>_S a(I_i).$$   
   For  $a\in T\cap S^{I_{|I|-1}+1}$,
   \begin{enumerate}\setcounter{enumi}{6}
  \item  choose $f(a)>_S a(I_{|I|-1})$ to be any $t$ such that for some $b\in T\cap S^{n+1}$ extending $a$, $t=a(I_{|I|-1})\up_\delta^\smallfrown b(n)\up_{[\delta, ht(b(n)))}$.
     \end{enumerate}
   Assume, by elementarity, that $f\in \mc{M}$.
   
Then we will inductively choose a $<_T$-increasing sequence $\langle a_i\in \mc{M}\cap T\cap S^{I_i+1}: i<|I|\rangle$ such that for any $i<|I|$,
  \begin{enumerate}\setcounter{enumi}{7}
  \item $f(a_i)<_S F_r(I_i)$;
  \item for any $v\in F_r$, if $v\up_{\mc{M}\cap \omega_1}=F_r(I_i)\up_{\mc{M}\cap \omega_1}$, then $c(a_i(I_i), v)=0$.
    \end{enumerate}
  
 We first find $a_0$. Note that $\langle F_r(j): j\leq I_0\rangle \in T$. So for some $\xi\in \mc{M}\cap \omega_1$, $\{a(I_0): a\in T\cap S^{I_0+1}\}$ is dense above $F_r(I_0)\up_\xi$. Applying Lemma \ref{modelcomb} to an appropriate  $\in$-subchain of $\ms{N}_r$, $\{v\in F_r: v\up_{\mc{M}\cap \omega_1}=F_r(I_0)\up_{\mc{M}\cap \omega_1}\}$ and $\{(a(I_0), f(a)): a\in T\cap S^{I_0+1}\}$, we get $a_0\in \mc{M}\cap T \cap S^{I_0+1}$ such that (8) and (9) hold for $i=0$.
 
 Now suppose that $a_k\in \mc{M}\cap T\cap S^{I_k+1}$ has been chosen to satisfy (8) and (9) and $k<|I|-1$. We check for $i=k+1$.  Let $x_k=F_r(I_{k+1})\up_\delta^\smallfrown f(a_k)\up_{[\delta, ht(f(a_k)))}$. By induction hypothesis (8) at $k$ and (1), 
 $$x_k<_S F_r(I_{k+1}).$$ 
 By  (6) at $k$, (5.3) and complete coherence of $S$, 
 $$\{b(I_{k+1}): b\in T\cap S^{I_{k+1}+1}\text{ extending }a_k\}\text{ is dense above }x_k.$$
 
 Applying Lemma \ref{modelcomb} to an appropriate  $\in$-subchain of $\ms{N}_r$, $\{v\in F_r: v\up_{\mc{M}\cap \omega_1}=F_r(I_{k+1})\up_{\mc{M}\cap \omega_1}\}$ and $\{(b(I_{k+1}), f(b)): b\in T\cap S^{I_{k+1}+1} $ extending $a_k\}$, we get $a_{k+1}\in \mc{M}\cap T \cap S^{I_{k+1}+1}$ extending $a_k$ such that (8) and (9) hold for $i=k+1$. This finishes the induction.\medskip
 
 Now suppose that $\langle a_i\in \mc{M}\cap T\cap S^{I_i+1}: i<|I|\rangle$ has been found to satisfy (8) and (9).  Then, by (7), find $b\in \mc{M}\cap T\cap S^{n+1}$ extending $a_{|I|-1}$ such that 
 $$b(n)\up_{[\delta, ht(b(n)))}=f(a_{|I|-1})\up_{[\delta, ht(f(a_{|I|-1})))}.$$
  Find $(p,t)\in D\cap \mc{M}$ witnessing (5) for $b$. It suffices to prove that $(p, t)$ is compatible with $(r, s)$.
 
 First note that by (5.1) and (8),  
 $$t\up_{[\delta, ht(t))}=b(n)\up_{[\delta, ht(b(n)))}=f(a_{|I|-1})\up_{[\delta, ht(f(a_{|I|-1})))}=F_r(|I|-1)\up_{[\delta, ht(t))}.$$
 Then by (5.3) and (1), $t<_S s$.
 
 Now we only need to check that $(\ms{N}_p\cup \ms{N}_r, F_p\cup F_r)$ is a condition of $\mc{P}$. The only non-trivial case is $(iii)$ for $u<_S v$ such that $u\in F_p$ and $v\in F_r$. Let $j<n$ be such that 
 $$u=F_p(j).$$
 
 Then by (1) and $t<_S s$,  
 $$u\up_{[\delta, ht(u))}=v\up_{[\delta, ht(u))}=s\up_{[\delta, ht(u))}=t\up_{[\delta, ht(u))}.$$
 By (5.4), $j\in I$. By (5.3), $u\up_\delta=F_r(j)\up_\delta$. Together with 
 (1) and the fact that $u<_S v$,  
 $$v\up_{\mc{M}\cap \omega_1}=F_r(j)\up_{\mc{M}\cap \omega_1}.$$
 Then by  (9) and the fact that $j\in I$, $c(u,v)=0$. 
 
 This shows that  $(\ms{N}_p\cup \ms{N}_r, F_p\cup F_r)$ is a condition  and finishes the proof of properness and $S$-preservation.\medskip
 
Above argument shows that for any $s\in S\setminus \mc{M}$, $(\{\mc{M}\cap H(\omega_2)\}, \{s\})$ is $(\mc{M}, \mc{P})$-generic and will force $\dot{F}$ to be uncountable where
$$\dot{F}=\bigcup\{F_p: p\in \dot{G}\} \text{ and $\dot{G}$ is the canonical name of the generic filter}.$$

Applying PFA$(S)$, we get an uncountable $F\subset S$ such that for any $u<_S v$ in $F$, $c(u,v)=0$. But this contradicts, e.g., Lemma \ref{lem1}.
 \end{proof}

\section{\texorpdfstring{MA$_{\omega_1}(S)[S]$}{MAomega1(S)[S]}  together with \texorpdfstring{$\mathrm{Pr}_0(\omega_1, \omega_1, \omega)$}{Pr0(omega1, omega1, omega)}}
In this section, we will introduce a new property and iterated force c.c.c. posets that preserve the property with finite support. Then force with $S$ to get a model of MA$_{\omega_1}(S)[S]$  together with $\mathrm{Pr}_1(\omega_1, \omega, \omega)$. Then a standard lifting argument shows that $\mathrm{Pr}_0(\omega_1, \omega_1, \omega)$ holds  (see, e.g., the beginning of Section 4 of \cite{st87} and  Lemma 1 of \cite{ss91}).

For completeness, we sketch the lifting argument here. Suppose that $d: [\omega_1]^2\ra \omega$ witnesses $\mathrm{Pr}_1(\omega_1, \omega, \omega)$. Fix a surjection $\pi_\alpha: \omega\ra \alpha$  for each $\alpha<\omega_1$ and a sequence of pairwise distinct reals $\langle r_\alpha\in 2^\omega: \alpha<\omega_1\rangle$. Let $\{h_n: n<\omega\}$ list all functions $h: 2^k\times 2^k\ra \omega$ where $k<\omega$. Define $d': [\omega_1]^2\ra \omega_1$   by 
$$d'(\alpha,\beta)=\pi_\beta(h_{d(\alpha,\beta)}(r_\alpha\up_k, r_\beta\up_k)$$
 where $k$ is the number such that  $dom(h_{d(\alpha,\beta)})=2^k\times 2^k$. Then $d'$ witnesses Pr$_0(\omega_1,\omega_1,\omega)$. To see this, fix $n<\omega$, uncountable pairwise disjoint $A\subset [\omega_1]^n$ and $H: n\times n\ra \omega_1$. Going to subsets, assume that for some $k$, $\langle r_{a(i)}\up_k: i<n\rangle$ is constant for $a\in A$ and consists of $n$ different elements. By $\mathrm{Pr}_1(\omega_1, \omega, \omega)$, choose $b\in A$ above $rang(H)$ such that  for any $m<\omega$,
 $$ \text{ there is }a\in A\cap [b(0)]^n\text{ such that }d[a\times b]=\{m\}. \eqno{(*)_m}$$
 Now choose $m$ such that $h_m(r_{b(i)}\up_k, r_{b(j)}\up_k)\in \pi_{b(j)}^{-1}\{ H(i,j)\}$ for any $i,j<n$. Then choose $a$ witnessing $(*)_m$. It is straightforward to check that $d'(a(i), b(j))=H(i,j)$ for any $i,j<n$.\medskip

We will start with a model of GCH in which there is a complete coherent Suslin tree $S$ and a coloring $c: [S]^2\ra \omega$ satisfying Pr$_0(S, \omega, \omega)$. 

Throughout this section, we use $\sigma, \tau$ to denote functions from $m$ to $2$ for some $m<\omega$ and we identify $2^m$ with the collection of functions from $m$ to $2$ for any $m<\omega$.  For $m<\omega$, we use ${\overset{\rightharpoonup}{0}}^m$  ($\overset{\rightharpoonup}{1}^m$) to denote the sequence of constant $0$ ($1$) with length $m$. We will omit $m$ if it is clear from the context.

Recall  the height order $<_{ho}$ in Definition \ref{ho}: for $\sigma\neq \tau$ in $2^m$ where $m<\omega$, $\sigma<_{ho} \tau$ if $\sigma(\Delta(\sigma, \tau))<\tau(\Delta(\sigma, \tau))$.

\begin{defn}\label{defc11}
For $0<n<\omega$ and $1<m<\omega$, $\mc{C}^1_{n,m}$ is the collection of all uncountable $A\subset ([S]^{n})^{2^m}$ such that for some $s\in S$, denoted by $stem(A)$, for any $\langle a_\sigma: \sigma\in 2^m\rangle\in A$,
\begin{enumerate}
\item  for any $\sigma\in 2^m$,  $s^\smallfrown \sigma<_S a_\sigma(0)<_S...<_S a_\sigma(n-1);$
\item for any $\sigma<_{ho} \tau$ in $2^m$,  $ht(a_\sigma(n-1))<ht(a_\tau(0))$;
\item for any $\sigma\neq \tau$ in $2^m$, $D_{a_\sigma(n-1), a_\tau(n-1)}=\{ht(s)+i: i\in D_{\sigma, \tau}\}$;
\item for any $\langle b_\sigma: \sigma\in 2^m\rangle\neq \langle a_\sigma: \sigma\in 2^m\rangle$ in $A$, 
$$\text{either } ht(a_{\overset{\rightharpoonup}{1}}(n-1))< ht(b_{\overset{\rightharpoonup}{0}}(0)) \text{ or } ht(b_{\overset{\rightharpoonup}{1}}(n-1))<ht(a_{\overset{\rightharpoonup}{0}}(0)).$$
\end{enumerate}
\end{defn}

\begin{defn}
\begin{enumerate}
\item $\mc{C}^1=\bigcup \{\mc{C}^1_{n,m}: 0<n<\omega,  1<m<\omega\}$.

\item For $A\in \mc{C}^1$, $(N_A, M_A)$ is the $(n,m)$ such that $A\in \mc{C}^1_{n,m}$. 
\end{enumerate}
\end{defn}
In this section, we will frequently take the greatest integer less than or equal to a logarithm with base 2.
\begin{defn}
For any real number $x$,   $\lfloor x \rfloor$ is the greatest integer less than or equal to $x$.
\end{defn}

Now, we are ready to define the   property $\varphi_1$.
\begin{defn}
\begin{enumerate}
\item For $0<n<\omega$ and $\langle A_i\in \mc{C}^1: i<n\rangle$, $\varphi_1(c, A_0,..., A_{n-1})$ is the assertion that for any collection $\langle s_\alpha: \alpha<\omega_1\rangle$ of pairwise distinct elements of $S$,  for any $\langle x^i_\alpha\in [A_i]^{<\omega}: i<n, \alpha<\omega_1\rangle$ such that $\langle x^i_\alpha: \alpha<\omega_1\rangle$ is pairwise disjoint for each $i<n$, there are $\alpha<\beta$ such that
\begin{enumerate}[$(i)$]
\item $s_\alpha<_S s_\beta$;
\item for any $i<n$, for any $\langle a_\sigma: \sigma\in 2^{M_{A_i}}\rangle\in x^i_\alpha$ and $\langle b_\sigma: \sigma\in 2^{M_{A_i}}\rangle\in x^i_\beta$, there is $\sigma\in 2^{M_{A_i}}$ such that for any $k, k'<N_{A_i}$,
$$c(a_\sigma(k), b_\sigma(k'))=\lfloor \log_2 M_{A_i}\rfloor -1.$$
\end{enumerate}

\item $\varphi_1(c)$ is the assertion that for any $0<n<\omega$, for any $\langle A_i\in \mc{C}^1: i<n\rangle$, $\varphi_1(c, A_0,..., A_{n-1})$ holds.
\end{enumerate}
\end{defn}

Clearly, $\varphi_1(c)$ implies that $S$ is Suslin. 

The procedure is similar as the procedure in Section 4. First, $\varphi_1(c)$ implies that $c$ induces an $S$ name of a witness for Pr$_1(\omega_1, \omega, \omega)$.
\begin{lem}\label{varphi1topr1}
Assume $\varphi_1(c)$. For any $n<\omega$, for any collection $\langle s_\alpha: \alpha<\omega_1\rangle$ of pairwise distinct elements of $S$, for any collection $\langle a_\alpha\in [ht(s_\alpha)]^n: \alpha< \omega_1\rangle$ of pairwise disjoint elements   and  $k<\omega$, there are $\alpha< \beta$ such that $s_\alpha<_S s_\beta $ and $c( s_\alpha\up_{a_\alpha(i)}, s_\beta\up_{a_\beta(j)})=k$ for any $i,j <n$.
\end{lem}
\begin{proof}
Fix $0<n<\omega$, pairwise distinct $\langle s_\alpha: \alpha<\omega_1\rangle$,  pairwise disjoint $\langle a_\alpha\in [ht(s_\alpha)]^n: \alpha< \omega_1\rangle$ and $k<\omega$. Fix $M$ such that 
$$\lfloor \log_2 M\rfloor -1=k.$$

By Fact \ref{f2}, fix $s\in S$ such that $\{s_\alpha: \alpha<\omega_1\}$ is dense above $s$. Repeating the argument in Lemma \ref{lem1}, we choose, for each $\alpha<\omega_1$,   $\langle f_\sigma(\alpha): \sigma\in 2^{M}\rangle$ such that
\begin{enumerate}
\item $A=\{\langle \{s_{f_\sigma(\alpha)}\up_{a_{f_\sigma(\alpha)}(i)}: i<n\}: \sigma\in 2^{M}\rangle : \alpha<\omega_1\}\in \mc{C}^1_{n, M}$;

\item $D_{s_{f_\sigma(\alpha)}, s_{f_\tau(\alpha)}}=\{ht(s)+i: i\in D_{\sigma, \tau}\}$ for any $\alpha<\omega_1$ and $\sigma, \tau\in 2^{M}$;

\item for any $\alpha<\omega_1$ and $\sigma\in 2^{M}$, $ht(s_{f_\sigma(\alpha)})\leq ht(s_{f_{\overset{\rightharpoonup}{1}}(\alpha)})$.
\end{enumerate}

Applying $\varphi_1(c)$ to $\langle s_{f_{\overset{\rightharpoonup}{1}}(\alpha)}: \alpha<\omega_1\rangle$ and $\langle \{\langle \{s_{f_\sigma(\alpha)}\up_{a_{f_\sigma(\alpha)}(i)}: i<n\}: \sigma\in 2^{M}\rangle\}: \alpha<\omega_1\rangle$, we get $\alpha<\beta$ and $\sigma\in 2^{M}$ such that
\begin{enumerate}\setcounter{enumi}{3}
\item $s_{f_{\overset{\rightharpoonup}{1}}(\alpha)}<_S s_{f_{\overset{\rightharpoonup}{1}}(\beta)}$;
\item for any $i,i'<n$,  $c(s_{f_\sigma(\alpha)}\up_{a_{f_\sigma(\alpha)}(i)}, s_{f_\sigma(\beta)}\up_{a_{f_\sigma(\beta)}(i')})=\lfloor \log_2 M\rfloor -1=k$.
\end{enumerate}

By (2)-(4),
\begin{enumerate}\setcounter{enumi}{5}
\item $s_{f_\sigma(\alpha)}<_S s_{f_\sigma(\beta)}$.
\end{enumerate}
It is clear that $f_\sigma(\alpha)<f_\sigma(\beta)$ are as desired.
\end{proof}

Now we prove that $\varphi_1(c)$  follows from Pr$_0(S, \omega, \omega)$.
\begin{lem}
If $c$ witnesses $\mathrm{Pr}_0(S, \omega, \omega)$, then $\varphi_1(c)$ holds.
\end{lem}
\begin{proof}
Arbitrarily choose $0<n<\omega$, $\langle A_i\in \mc{C}^1: i<n\rangle$,  a collection $\langle s_\alpha: \alpha<\omega_1\rangle$ of pairwise distinct elements of $S$,  a collection $\langle x^i_\alpha\in [A_i]^{<\omega}: i<n, \alpha<\omega_1\rangle$ such that $\langle x^i_\alpha: \alpha<\omega_1\rangle$ is pairwise disjoint for each $i<n$. It suffices to find $\alpha<\beta$ witnessing the assertion of $\varphi_1(c, A_0,..., A_{n-1})$.

%Going to uncountable subsets, we may assume that for some $\langle m_i: i<n\rangle$, $|x^i_\alpha|=m_i$ for any $i<n$ and $\alpha<\omega_1$.  

Now for each $\alpha<\omega_1$,  define
$$F_\alpha=\{a_\sigma(k):  \exists i<n \  \exists\langle a_\tau: \tau\in 2^{M_{A_i}}\rangle\in x^i_\alpha\ (\sigma\in 2^{M_{A_i}}  \wedge k<N_{A_i})\} .$$
%$a_\alpha\in S^{|F_\alpha|}$ by $a_\alpha(k)=$ the $k$th element of $F_\alpha$ in $<_{ho}$-increasing enumeration.

%Since extending $s_\alpha$ does not change the conclusion, we may assume that  for any $\alpha<\omega_1$,
%$$ht(s_\alpha)>\max\{ht(s): s\in F_\alpha\}.$$

Denote $D=\bigcup_{i<n} [ht(stem(A_i)), ht(stem(A_i))+M_{A_i})$.

Going to uncountable subsets, we may assume that 
\begin{enumerate}
\item $\{F_\alpha: \alpha<\omega_1\}$ is pairwise disjoint;
\item for any $\alpha<\omega_1$, $|F_\alpha|=|F_0|$ and $\max D< \min\{ht(s): s\in F_\alpha\}$;
%\item for any $i<n$ and $j<m_i$, the position of $x^i_\alpha(j)$ in $F_\alpha$ is independent of $\alpha$, i.e., for any $k<2N_{A_i}$ and $l<|F_\alpha|$,  $x^i_\alpha(j)(k)=F_\alpha(l)$ iff $x^i_0(j)(k)=F_0(l)$;
\item for any $\alpha<\omega_1$ and $k<|F_0|$, $F_\alpha(k)\up_D=F_0(k)\up_D$.
\end{enumerate}

Now define $h: |F_0|\ra \omega$ by  for any $k<|F_0|$,
 $$h(k)=l \text{ iff } \exists j \  |\{\xi\in D: F_0(k)(\xi)=0\}| =2^l(2j+1)-1.$$ 
 Apply Pr$_0(S, \omega, \omega)$ to get $\alpha<\beta$ such that
\begin{enumerate}\setcounter{enumi}{3}
\item $s_\alpha<_S s_\beta$;
\item for any $k,k'< |F_0|$, $c(F_\alpha(k), F_\beta(k'))=h (k)$.
\end{enumerate}
Together with (2), (3) and the definition of $h$, we conclude that for any $k,k'< |F_\alpha|$, 
\begin{enumerate}\setcounter{enumi}{5}
\item $c(F_\alpha(k), F_\beta(k'))=l$ iff $\exists j  \ |\{\xi\in D: F_\alpha(k)(\xi)=0\}| =2^l(2j+1)-1$.
\end{enumerate}\medskip

We now check that $\alpha<\beta$ are as desired. Fix $i<n$, $\langle a_\sigma: \sigma\in 2^{M_{A_i}}\rangle\in x^i_\alpha$ and $\langle b_\sigma: \sigma\in 2^{M_{A_i}}\rangle\in x^i_\beta$.  Note that   
$$\{|\{\xi\in D: a_\sigma(0)(\xi)=0\}|: \sigma\in 2^{M_{A_i}}\}=[K, K+M_{A_i}]$$
where $K=|\{\xi\in D: a_{\overset{\rightharpoonup}{1}}(0)(\xi)=0\}|$.

Then there exists $\sigma\in 2^{M_{A_i}}$ such that for some $j$,
 $$|\{\xi\in D: a_\sigma(0)(\xi)=0\}|=2^{\lfloor \log_2 M_{A_i}\rfloor-1}(2j+1)-1.$$
 
 By (6) and the definition of $\mc{C}^1$,   for any $k, k'<N_{A_i}$,
$$c(a_\sigma(k), b_\sigma(k'))=\lfloor \log_2 M_{A_i}\rfloor -1.$$
 This, together with (4),  shows that $\alpha<\beta$ are as desired and finishes the proof of the lemma.
\end{proof}

 The proof of Lemma \ref{iteration} shows the following.
 \begin{lem}\label{iteration1}
Suppose that $\nu$ is a limit ordinal and $\mc{P}_\nu$ is the direct limit of the finite support iteration  $\langle \mc{P}_\beta, \dot{\mc{Q}}_\beta: \beta< \nu\rangle$   of c.c.c.  posets. If for any $\beta<\nu$, $\Vdash_{\mc{P}_\beta} \varphi_1(c)$, then $\Vdash_{\mc{P}_\nu} \varphi_1(c)$.
\end{lem}

We define the standard forcing to destroy c.c.c. of posets that destroy $\varphi_1(c)$.
\begin{defn}
Suppose that $0<n<\omega$, $A_0,..., A_{n-1}$ are in $\mc{C}^1$ and $\langle x^i_\alpha\in [A_i]^{<\omega}:\alpha<\omega_1\rangle$ is pairwise disjoint for each $i<n$. Then $\mc{Q}^1_{\langle x^i_\alpha: i<n,\alpha<\omega_1\rangle}$ is the poset consisting of all $p\in [\omega_1]^{<\omega}$ such that
\begin{itemize}
\item  for any $\alpha<\beta$ in $p$ and $i<n$, for any  $\langle a_\sigma: \sigma\in 2^{M_{A_i}}\rangle \in x^i_\alpha$ and  $\langle b_\sigma: \sigma\in 2^{M_{A_i}}\rangle\in x^i_\beta$,  there is $\sigma\in 2^{M_{A_i}}$ such that for any $k,k'<N_{A_i}$, 
$$c(a_\sigma(k), b_\sigma(k'))=\lfloor \log_2 M_{A_i} \rfloor -1.$$
\end{itemize}
The order is reverse inclusion.
\end{defn}
The proof of Lemma \ref{preservation}  shows that   $\varphi_1(c)$ is preserved by $\mc{Q}^1_{\langle x^i_\alpha: i<n,\alpha<\omega_1\rangle}$.
\begin{lem}
Assume $\varphi_1(c)$.
If $0<n<\omega$, $A_0,..., A_{n-1}$ are in $\mc{C}^1$ and $\langle x^i_\alpha\in [A_i]^{<\omega}:\alpha<\omega_1\rangle$ is pairwise disjoint for each $i<n$, then $\mc{Q}^1_{\langle x^i_\alpha: i<n,\alpha<\omega_1\rangle}$ is c.c.c. and $\Vdash_{\mc{Q}^1_{\langle x^i_\alpha: i<n,\alpha<\omega_1\rangle}} \varphi_1(c)$.
\end{lem}

Now the argument in Theorem \ref{varphi0} shows the following.
 \begin{thm}\label{varphi1}
It is consistent that there is a complete coherent Suslin tree $S$ and a coloring $c: [S]^2\ra \omega$ such that $\mathrm{MA}_{\omega_1}(S)$ and $\varphi_1(c)$ hold.
 \end{thm}
 
 Now forcing with $S$ over the model mentioned in the theorem above and using the lifting argument described in the beginning of the section, we conclude the following.
 \begin{cor}
$\mathrm{MA}_{\omega_1}(S)[S]$ is consistent with $\mathrm{Pr}_0(\omega_1, \omega_1, \omega)$.
\end{cor}

We formulate  a new forcing axiom similar as MA$_{\omega_1}(S, c, \varphi_0)$.
\begin{defn}
For a complete coherent Suslin tree $S\subset 2^{<\omega_1}$ and a coloring $c: [S]^2\ra \omega$, $\mathrm{MA}_{\omega_1}(S, c, \varphi_1)$ is the assertion that
\begin{enumerate}[(i)]
\item $\varphi_1(c)$ holds;
\item if $\mc{P}$ is a c.c.c. forcing that preserves $\varphi_1(c)$ and $\{D_\alpha: \alpha<\omega_1\}$ is a collection of dense subsets of $\mc{P}$, then there is a filter meeting them all.
\end{enumerate}
\end{defn}
Say that MA$_{\omega_1}(S, c, \varphi_1)[S]$  holds if  the universe is a forcing extension by $S$ over a model of MA$_{\omega_1}(S, c, \varphi_1)$.

The argument of Proposition \ref{prop} shows the following.
\begin{prop}
For a complete coherent Suslin tree $S\subset 2^{<\omega_1}$ and a coloring $c: [S]^2\ra \omega$, the following are equivalent.
\begin{enumerate}
\item $\mathrm{MA}_{\omega_1}(S, c, \varphi_1)$.
\item $\mathrm{MA}_{\omega_1}(S)$ and $\varphi_1(c)$.
\end{enumerate}
\end{prop}

\begin{cor}
\begin{enumerate}
\item $\mathrm{MA}_{\omega_1}(S, c, \varphi_1)$ is consistent.
\item $\mathrm{MA}_{\omega_1}(S, c, \varphi_1)[S]$ implies $\mathrm{MA}_{\omega_1}(S)[S]$ and $\mathrm{Pr}_0(\omega_1, \omega_1, \omega)$.
\end{enumerate}
\end{cor}

\bibliographystyle{plain}

\end{document}